\documentclass[11 pt,oneside]{article}
\usepackage{amsmath}
\usepackage{amsxtra}
\usepackage{amscd}
\usepackage{amsthm}
\usepackage{amsfonts}
\usepackage{amssymb}
\usepackage{eucal}
\usepackage{graphicx}
\usepackage{float}
\usepackage{wrapfig}
\graphicspath{ {images/} }

\usepackage[margin=2.5 cm, reversemp]{geometry}
\usepackage{marginnote}

\usepackage{diagrams}

\theoremstyle{definition}
\newtheorem{thm}{Theorem}[section]
\newtheorem{cor}[thm]{Corollary}
\newtheorem{lemma}[thm]{Lemma}
\newtheorem{propn}[thm]{Proposition}
\newtheorem{defn}[thm]{Definition}

\newtheorem*{rmk}{Remark}

\newtheorem*{claim}{Claim}

\newcommand\nc{\newcommand}
\nc\Span{\text{\rm Span}}
\nc\Id{\text{Id}}

\nc\rank{\text{rank}}

\nc \cc {\mathbb{C}}
\nc \ff {\mathbb{F}}
\nc \hh {\mathbb{H}}
\nc \nn {\mathbb{N}}
\nc \pp {\mathbb{P}}
\nc \qq {\mathbb{Q}}
\nc \rr {\mathbb{R}}
\nc \zz {\mathbb{Z}}

\nc\cA{\mathcal{A}}
\nc\cB{\mathcal{B}}
\nc\cC{\mathcal{C}}
\nc\cD{\mathcal{D}}
\nc\cF{\mathcal{F}}
\nc\cG{\mathcal{G}}
\nc\cH{\mathcal{H}}
\nc\cK{\mathcal{K}}
\nc\cL{\mathcal{L}}
\nc\cO{\mathcal{O}}
\nc\cP{\mathcal{P}}
\nc\cS{\mathcal{S}}
\nc\csf{\mathcal{S}\mathcal{F}}

\nc\fm{\mathfrak{m}}
\nc\fp{\mathfrak{p}}

\nc\Tr{\text{Tr}}
\nc\into{\hookrightarrow}

\nc\st{\text{ s.t. }}
\nc\intense[1]{\textcolor[rgb]{1.00,0.00,0.00}{\textbf{#1}}}

\renewcommand{\(}{\left(}
\renewcommand{\)}{\right)}

\nc\Mat{\text{\rm Mat}}
\nc\GL{\text{\rm GL}}
\nc\SU{\text{\rm SU}}
\nc\SO{\text{\rm SO}}
\nc\SL{\text{\rm SL}}
\nc\Sp{\text{\rm Sp}}
\nc\EL{\text{\rm EL}}

\nc\GEM{\text{\rm GEM}}
\nc\Alt{\text{\rm Alt}}
\nc\Sym{\text{\rm Sym}}

\renewcommand{\(}{\left(}
\renewcommand{\)}{\right)}

\nc\inject{\hookrightarrow}

\nc{\mattwo}[4]{\left[\begin{array}{cc} #1  & #2\\  #3 & #4 \\ \end{array} \right]}
\nc{\matthree}[9]{\left[\begin{array}{ccc} #1  & #2 & #3\\  #4 & #5 & #6 \\ #7 & #8 & #9 \\ \end{array} \right]}
\nc{\vecttwo}[2]{\left[\begin{array}{c} #1 \\ #2 \\ \end{array} \right]}
\nc{\vectthree}[3]{\left[\begin{array}{c} #1 \\ #2\\  #3 \\ \end{array} \right]}

\nc{\del}{\partial}
\nc\onto{\twoheadrightarrow}

\nc\const{\text{const}}

\nc\rrp{\mathbb{RP}}
\nc\ul{\underline}
\nc\ol{\overline}
\nc\uline{\underline}
\nc\oline{\overline}
\nc\oset{\overset}
\nc\uset{\underset}

\nc\heart{\heartsuit}
\nc\spade{\spadesuit}
\nc\club{\clubsuite}

\nc\marg[1]{\marginnote{\boxed{\text{#1}}}}
\nc\margq[1]{\marginnote{\textcolor[rgb]{1.00,0.00,0.00}{#1}}}

\nc\Char{\text{Char}}
\nc\Frac{\text{Frac}}
\nc\wo{\backslash}
\nc\diag{\text{diag}}
\nc\wtl{\widetilde}
\nc\nsubgp{\triangleleft}

\nc\Cay{\text{Cay}}
\nc\Hom{\text{Hom}}
\nc\Gp{\text{Gp}}
\nc\Set{\text{Set}}
\nc\la{\langle}
\nc\ra{\rangle}

\nc\Spec{\text{Spec}}
\nc\ad{\text{ad}}
\nc\Tor{\text{Tor}}

\nc\wht{\widehat}
\nc\ddx[2]{\frac{\partial {#1}}{\partial {#2}}}
\nc\dddx[3]{\frac{\partial^2 {#1}}{\partial {#2}\partial{#3}}}
\nc\mult{\text{mult}}
\nc\supp{\text{supp}}
\nc\sign{\text{sign}}

\nc\tr{\text{tr}}
\nc\stab{\text{stab}}
\nc\im{\text{im}}

\nc\sech{\text{sech}}

\nc\Ind{\text{Ind}}
\nc\tsf{\text{sf}}

\nc\Kh{\text{Kh}}
\nc\red{\text{red}}

\title{Khovanov homology and binary dihedral representations for marked links}
\author{Sherry Gong}

\begin{document}
\maketitle

\begin{abstract} We introduce a version of Khovanov homology for alternating links with marking data, $\omega$, inspired by instanton theory. We show that the analogue of the spectral sequence from Khovanov homology to singular instanton homology introduced in \cite{KM_unknot} for this marked Khovanov homology collapses on the $E_2$ page for alternating links. We moreover show that for non-split links the Khovanov homology we introduce for alternating links does not depend on $\omega$; thus, the instanton homology also does not depend on $\omega$ for non-split alternating links.

Finally, we study a version of binary dihedral representations for links with markings, and show that for links of non-zero determinant, this also does not depend on $\omega$.
\end{abstract}

\section{Introduction}

Throughout this paper, we shall work over a field of characteristic $2$.

Let $L \subset S^3$ be a link, and let $\omega$ be a one dimensional submanifold of $S^3$ with boundary in $L$, thought of as the Poincare dual of $w_2(Q)$, where $Q$ is an $SO(3)$ bundle on the link complement, $S^3 \backslash L$. 

In \cite{KM_unknot}, Kronheimer and Mrowka introduced a singular instanton homology $I(L,\omega)$ for a link $L$ with singular bundle data given by $\omega$, and they constructed a spectral sequence with $E_2$ page the Khovanov homology of the link, which abuts to $I^\#(L, \emptyset) = I(L \cup H, \omega_0)$, that is, the instanton homology of $L$ with a Hopf link, $H$, at infinity, and $\omega_0$ a single arc between the two components of $H$. They further show that for alternating knots, the spectral sequence collapses on the $E_2$ page.

For $P_L$ a projection of the link $L$, Kronheimer and Mrowka's spectral sequence can be generalised to all $\omega$, so that it becomes a spectral sequence whose $E_2$ page is an object we call $H(P_L,\partial \omega)$, and which abuts to $I^\#(L, \omega) = I(L \cup H, \omega_0 \cup \omega)$, where $H(P_L,\partial \omega)$ is constructed as follows.

Let $A$ denote the $\zz/2$-algebra $\zz/2[x]/x^2$. Consider the cube of resolutions of of a link projection $P_L$ with $n$ crossings, where each vertex $v \in \{0,1\}^n$ of the cube is assigned a resolution $D_v$, by resolving each crossing as in Figure \ref{resolutions_fig}.

\begin{figure}[ht!]
\centering
\includegraphics[width=70mm]{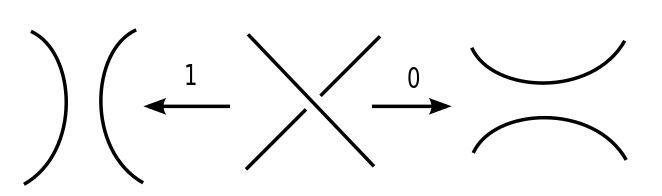}
\caption{}
\label{resolutions_fig}
\end{figure}

\begin{wrapfigure}{R}{35mm}
\centering
\includegraphics[width=25mm]{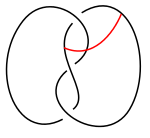}
\label{trefoil_with_omega}
\end{wrapfigure}
At each vertex of the cube, we then have an unlink and some marked points on the unlink representing $\partial \omega$, the boundary of $\omega$. Let $C(P_L,\partial \omega)$ be the complex that assigns to a resolution of $c$ components $A^{\otimes c}$ if each of the components has an even number of endpoints of $\omega$, and otherwise assigns that resolution $0$. For example, for the trefoil with $\omega$ to the right, the resolutions are as given in Figure \ref{trefoil_omega_resn}. The only resolutions whose unlink has more than one component for which the $A^{\otimes c}$ has not been replaced with $0$ is the $(1,0,0)$ resolution, because for that one the arc $\omega$ has endpoints on the same component of the unlink.

\begin{figure}[ht!]
\centering
\includegraphics[width=90mm]{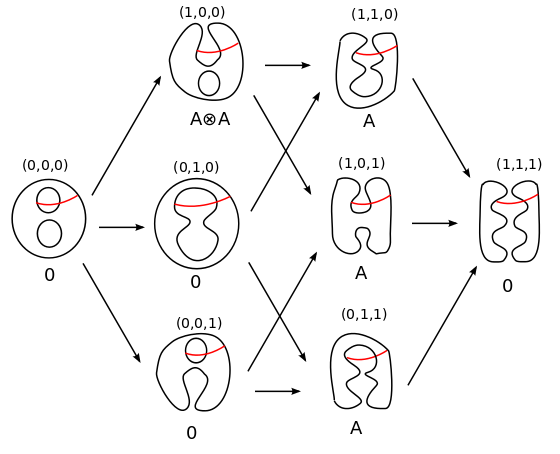}
\caption{}
\label{trefoil_omega_resn}
\end{figure}

The differentials in $C(P_L,\partial \omega)$ are along edges of the cube as depicted in figure \ref{trefoil_omega_resn} and given by the merge map $m:A \otimes A \to A$ and the split map $\Delta:A \to A \otimes A$, where
\[m(1 \otimes 1) = 1 \text{, } m(1 \otimes x) = m(x \otimes 1) = x \text{, } m(x\otimes x) = 0,\]
and 
\[\Delta(1) = 1 \otimes x + x \otimes 1 \text{, } \Delta(x) = x \otimes x,\]
when the source and target are both non-zero. 

\begin{defn} The marked Khovanov homology of $P_L$ with $\omega$, $H(P_L, \partial \omega)$ is the homology of the complex $C(P_L,\partial\omega)$. 
\end{defn}

For a link projection $P_L$ with marking data $\omega$ and a basepoint $p \in P_L$, we may also consider the reduced complex $C_{\red}(P_L,\partial\omega)$ formed by, at each vertex, replacing the $A$ in $A^{\otimes n}$ corresponding to the component with the basepoint with $A/\la x \ra$, similarly to the usual reduced Khovanov complex. The differentials in this complex are defined similarly, replacing $m$ and $\Delta$ with the induced maps $m_{\red}$ and $\Delta_{\red}$ on the quotients. The following definition is the reduced version of the above.

\begin{defn} The reduced marked Khovanov homology of $P_L$ with $\omega$, $H_{\red}(P_L, \partial \omega)$ is the homology of the complex $C_{\red}(P_L,\partial\omega)$. 
\end{defn}

In general, $H(P_L,\partial\omega)$ is not a marked link invariant, in that it is not invariant with respect to moving an endpoint of $\omega$ along a component, nor is it invariant with respect to Reidermeister II and III moves.

In this paper, however, we show an invariance result for alternating link projections.

\begin{thm} For an alternating projection $P_L$ of an alternating link $L$, with marking data $\omega$, $H(P_L,\partial \omega)$ is a marked link invariant; that is, it is invariant with respect to different projections for the same alternating link and with respect to moving an endpoint of $\omega$ along a component of the link. Moreover, for non-split alternating links, it does not depend on $\omega$. For based links, the same is true for $H_\red(P_L,\partial \omega)$.
\label{naive_khov_invar}
\end{thm}

Analogously to the spectral sequence in \cite{KM_unknot}, there is a spectral sequence from $H(P_L,\partial\omega)$ to $I^\#(L, \omega)$. In \cite{KM_unknot}, they show that the spectral sequence collapses for quasi-alternating knots $K$. We extend that to a result for alternating link projections $L$ with marking data $\omega$:

\begin{thm} For alternating link projections $P_L$, the spectral sequence from $H(P_L,\partial\omega)$ to $I^\#(L, \omega)$ collapses on the $E_2$ page.
\end{thm}

Combining this theorem with Theorem \ref{naive_khov_invar}, we have:

\begin{cor} For non-split alternating links $L$, the instanton homology $I^\#(L, \omega)$ does not depend on $\omega$.
\end{cor}

In \cite{KM_filtrations}, Kronheimer and Mrowka also exhibited filtrations $q$ and $h$ on the Khovanov complex for an alternating link such that the Khovanov differential increases $h$ by $1$ and preserves $q$, and such that the difference between the instanton differential and the Khovanov differential has order $ \geq 1$ with respect to the $h$ filtration and $\geq 2$ with respect to the $q$ filtration. They moreover use this to show that the isomorphism types of the pages of the spectral sequence with respect to the $q$ and $h$ filtrations are link invariants.

We extend the $q$ filtration result to links with certain $\omega$. Specifically, consider $\omega$ corresponding to singular bundle data $P_\Delta$ satisfying that on the cobordism corresponding to each diagonal of the cube, we have $\cP(w_2(P_\Delta)) \equiv 0 \pmod{4}$ where $\cP$ is the Pontrjagin square. Here $P_\Delta$ is defined in section \ref{q_section}, as in \cite{KM_unknot}, to be a certain principal $PU(2)$-bundle on a non-Hausdorff space $X_\Delta$ coming from $(X,\Sigma)$ where $\Sigma$ is the cobordism and $X$ is the ambient space, $S^3 \times \rr$. 

For such $\omega$, we define a $q$ filtration on the modified Khovanov complex, so that the instanton differential has order $\geq 0$. We use this to show that the isomorphism class of the first page of the spectral sequence from $H(P_L,\omega)$ to the instanton homology is a tangle invariant of the tangle obtained by considering the part of the link outside of a ball containing $\omega$.

In \cite{SS_tqa}, Scaduto and Stoffregen studied the homology of the complex $C(P_L,\partial \omega)$, which they called $Hd(D,\omega)$, and exhibited its relation via a spectral sequence to the framed instanton homology of the double branched cover of the link. This spectral sequence is the framed instanton theory analogue of the spectral sequence in \cite{KM_unknot}. They moreover conjecture a relation between $Hd(D,\omega)$ and a twisted Khovanov homology similar to those in \cite{BLS}, \cite{Jae13}, and \cite{Rob15}, which is an invariant of links with marking data, and which also has a spectral sequence relating it to the framed instanton homology.

We will also look at modifying the space of binary dihedral representations to account for $\omega$: recall that the binary dihedral group $BD \subset SU(2) \simeq S^3 \subset \hh$ is given by $BD = S^1_A \cup S^1_B$, where $S^1_A = e^{I\theta}$ and $S^1_B = Je^{I\theta}$. 

Recall that a dihedral subgroup of $SO(3)$ is a a group generated by rotations about a fixed axis and reflections about the orthogonal plane to that axis, and a binary dihedral representation $\rho:G \to SU(2)$ is a representation whose image in $SO(3)$ via the canonical map $SU(2) \to SO(3)$ is contained in a dihedral subgroup.

The space of binary dihedral representations of the fundamental group of a link complement, which are conjugate to representations $\rho:\pi_1(S^3 \backslash K) \to BD \subset SU(2)$, has been studied as a link invariant. In \cite{Klassen}, Klassen showed that for $\Gamma = \pi_1(S^3-K)$ for a knot $K$, the number of conjugacy classes of non-abelian homomorphisms $\Gamma \to BD$ is
\[(|\Delta_K(-1)|-1)/2\]
where $\Delta_K(t)$ is the Alexander polynomial of $K$.

In \cite{Zentner}, Zentner studied knots with the property that all of its $SU(2)$ representations are binary dihedral and called such knots ``$SU(2)$-simple''. He showed that if a knot $K$ is $SU(2)$-simple and satisfies a certain genericity hypothesis, then the higher differentials on the instanton complex vanish.

In this case, we study a modification of the space of binary dihedral representations for links with $\omega$. Note that all the meridians of each component of $L$ are conjugate to each other in $\pi_1(S^3 \backslash L)$. Moreover, elements of $S^1_B$ can only be conjugate to other elements of $S^1_B$, so either all meridians of a given component of $L$ go to $S^1_B$, or they all go to $S^1_A$. For the representation to be non-abelian, they must go to $S^1_B$ for at least one component.

To modify the link invariant of binary dihedral representations to account for $\omega$ in the spirit of the representations spaces that arise in instanton homology, we consider the space of representations of $\pi_1(S^3 \backslash (L \cup \omega))$ which take the meridians around $\omega$ to $-1$.

We will primarily want to consider the representations which map the meridians around the link components to $S^1_B$. Let $L = \cup L_i$ be the components of $L_i$.

\begin{defn} For a link $L$, let the spaces of marked binary dihedral representations modulo conjugation be denoted by
\[R(L,\omega) = \{\rho:\pi_1(S^3 \backslash (L \cup \omega)) \to BD | \rho(\mu_\omega)=-1 \in S^1_A\}/\text{conj}\]
and
\[R_B(L,\omega) = \{\rho:\pi_1(S^3 \backslash (L \cup \omega)) \to BD |\rho(\mu_{L_i}) \in S^1_B \text{, } \rho(\mu_\omega) = -1 \in S^1_A\}/\text{conj}.\]
where $\mu_{L_i}$ is a meridian around $L_i$ and $\mu_\omega$ is a meridian around $\omega$.
\end{defn}

These are marked link invariants. That is,

\begin{lemma}
The dependence on $\omega$ of the spaces $R(L,\omega)$ and $R_B(L,\omega)$ can be reduced to the parity of the number of endpoints of $\omega$ on each component.
\end{lemma}

In particular, if $L$ is a knot, then these invariants do not depend on $\omega$.

We shall prove that similarly to the Khovanov homology we defined, for non-split alternating links, and more generally, for links of non-zero determinant, this invariant does not depend on $\omega$:

\begin{thm} For a link $L$ with non-zero determinant and singular bundle data $\omega$, the number of conjugacy classes of binary dihedral representations in $R_B(L,\omega)$ does not depend on $\omega$.
\end{thm}

We will also show a partial converse to this: For a link $L$ with determinant zero, the number of conjugacy classes in $R_B(L,\omega)$ does depend on $\omega$. In particular, $R_B(L,\emptyset) \neq \emptyset$, but we will show that there is $\omega$ such that $R_B(L,\omega)$ is empty.


\section{Marked points Khovanov homology}

Given a link projection $P_L$, with a finite set of marked points $\partial \omega$, recall in the introduction, we defined a complex $C(P_L,\partial\omega)$, which was like the Khovanov complex, but with $0$ instead of $A^{\otimes n}$ at vertices of the cube where a component has an odd number of marked points. In the latter case, where a component has an odd number of marked points, we say that $\partial \omega$ ``kills'' the vertex in the modified Khovanov complex.

\begin{lemma} The $C(P_L,\partial\omega)$ defined above is actually a complex, ie, $d^2=0$.
\end{lemma}
\begin{proof} Just as in the usual Khovanov homology, we need only show that squares in the cube commute.
\begin{diagram}
D_{10} & \rTo & D_{11}\\
\uTo & & \uTo \\
D_{00} & \rTo & D_{01}\\
\end{diagram}
If $\partial \omega$ does not kill any of the corners in the square, then the edge maps are the same as those in the usual Khovanov complex, so the square commutes. If $D_{00}$ or $D_{11}$ is killed, or if $D_{10}$ and $D_{01}$ are both killed, then the square obviously commutes. 

The only remaining case is that $D_{00}$ and $D_{11}$ are both not killed, but one of $D_{10}$ and $D_{01}$ is killed. Without loss of generality assume that it is $D_{10}$. Then there must be a two components in the $D_{10}$ diagram with an odd number of marked points each, and these two components must be merged into one component in both $D_{00}$ and $D_{11}$. This is only possible when the square is the projection of a two component unlink that has two crossings between the components, corresponding to the two dimensions of the square. 

Thus, the other map $D_{00} \to D_{01} \to D_{11}$ is $m \circ \Delta$, which is $0$, because
\[m(\Delta(1))=m(1 \otimes x+x \otimes 1)=0,\]
because we are over a field of characteristic 2, and 
\[m(\Delta(x))=m(x \otimes x) =0.\]

\end{proof}

At this point, we have not assumed that the projection is alternating, but it already makes sense to consider the spectral sequence from $H(P_L,\partial \omega)$ to $I^\#(L,\omega)$ analogous to the one in \cite{KM_unknot}. However, $H(P_L,\partial \omega)$ is not an independent of the choice of projection for $L$, nor of choice of where the endpoints of $\omega$ are on the components. For a counterexample to the latter, see Figure \ref{non_alt_eg}. 

\begin{figure}[ht!]
\centering
\includegraphics[width=60mm]{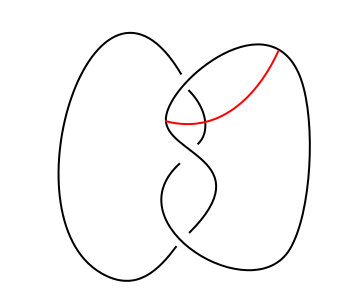}
\caption{In this non-alternating projection for an unknot, with $\omega$ being the arc shown, it is easy to see that $H(P_L,\partial \omega)$ has dimension $6$ over $\zz/2\zz$. However, if we used a projection with no crossings, we would get dimension $2$.}
\label{non_alt_eg}
\end{figure}

\subsection{Marked points and alternating links}
\label{alternating_marked_points_section}

We have now seen an example that shows that our marked point Khovanov homology may change as an endpoint of $\omega$ slides along a component in the link. We shall see, however, that this cannot happen for alternating link projections:

\begin{propn}
For an alternating link projection, the complex above is well defined up to not-necessarily-degree-preserving quasi-isomorphism. That is the operations of sliding an endpoint across a crossing does not change the complex (up to not-necessarily-degree-preserving quasi-isomorphism).
\label{propn well def}
\end{propn}

To that end, first let us show that for alternating link projections, we can compute our homology dropping one of the crossings.

\begin{rmk}Throughout this discourse in all diagrams, the ``vertical maps'' will always correspond to the one crossing we are trying to drop.
\end{rmk}

\begin{defn}

Note that if we `drop' one crossing in a cube a resolutions, ie, leave it unresolved, we still have unlinks, because a projection with only one crossing can only be an unlink. We will call a partial resolution that is an unlink a ``pseudo-diagram resolution''.

For crossing $x$ on the alternating link projection $P_L$ with marking data $\partial \omega$, we define a complex $C(P_L,\partial \omega, x)$ whose underlying groups are the same as before: if there are $k$ crossings total, form the $k-1$ dimensional cube of resolutions from resolving all crossings except $x$. Place $A^{\otimes n}$ at a vertex where there are $n$ components to the unlink there, unless some component has an odd number of markings, in which case place $0$. These are the chain groups.

Let us define the differential $d_C$. There is a map for each edge, and edges correspond to crossings, so let us define the differential corresponding to edge $y$ thus:
\begin{itemize}
\item Type 0: The number of components in the pseudodiagram resolution changes and neither source nor target is killed by $\omega$: In this case take the maps to be $m$ or $\Delta$, the merging or splitting maps of the Khovanov complex.
\item Type 1: The number of components of the pseudodiagram resolution changes and at least one of source or target is killed by $\omega$: In this case the map is 0.
\item Type 2: The number of components of the pseudodiagram resolution does not change. In this case, the map is $0$.
\end{itemize}
\label{drop cpx}
\end{defn}

\begin{lemma} For a link projection, the cube with a dropped crossing $(C,d_C)$ defined above forms a complex; since we are working over $\zz/2$, this is saying that the cube commutes.
\label{dropped_commutes}
\end{lemma}
\begin{proof} Consider a square in $C$:
\begin{diagram}
C_2 & \rTo & C_4\\
\uTo & & \uTo \\
C_1 & \rTo & C_3\\ 	
\end{diagram}
There are two crossings that are being resolved in this square (in addition to the crossing left unresolved). Let us consider only the \textit{active components}, that is the components in the pseudo-diagram resolutions in question that involve at least one of the crossings. Let $a_i$ be the number of active components in the pseudodiagram resolution corresponding to $C_i$.

Note that the minimal $a_i$ is at most two, because in the corner of the square with the minimal $a_i$, each crossing can only involve one component.

If the total number of endpoints of $\omega$ on these active components is odd, then all the $C_i$ are 0, so the diagram commutes. We assume that the total number of endpoints on the active components is even.

Moreover, if the unresolved crossing is not in one of the active components, then the diagram commutes because it looks the same as the marked Khovanov for a fully resolved link, which we showed commutes above. Thus, we may assume that the unresolved crossing is on an active component.

We do some casework:
\begin{enumerate}
\item If $\min(a_i)=2$. Consider the corner with minimal $a_i$. It has two components, both of which are active. There cannot be a crossing that goes between components, or resolving it the other way would lead to lower $a_i$. Thus, there must be one crossing on each of the two components. In this case it is clear that the diagram commutes, because the two crossings act independently of each other and are then tensored together.

For all other cases, $\min(a_i)=1$. 

\item If $\min(a_i)=1$ and $\max(a_i)=3$. In this case all the maps change number of components, so they are $m$, $\Delta$, or $0$, where they are only $0$ if either the source or the target is $0$. Then all maps are analogous to the case where we did not drop a crossing.

In particular, at the corner with $a_i = 1$, the two resolved crossings must each be restricted to one wing of the component, which means that if we resolve the remaining crossing in a way that doesn't change the number of components at the vertex with $a_i=1$, we do not affect the other groups and morphisms in the diagram. Commutativity now follows from commutativity for the case with all crossings resolved.

\item If $\min(a_i) = 1$ and $\max(a_i)=2$.

If adjacent vertices on the cube all have different $a_i$, then we get commutativity for the same reason that commutativity works in the case where we do not drop a crossing, since all the maps are analogous, as in the previous case.

So we may assume that there is either a pair of adjacent vertices with $a_i=1$, or there is a pair with $a_i=2$.

If there are simultaneously a pair with $a_i=1$ and a pair with $a_i=2$, then it is clear that however you traverse the square, you get zero, so it commutes. 

In there aren't both pairs simultaneously, then the $a_i$ are either $(1,1,1,2)$ in some order, or $(2,2,2,1)$ in some order. Since $m \circ \Delta =0$, the case $(1,1,1,2)$ commutes.

The case $(2,2,2,1)$ is not actually possible: consider the corner with one component. This is a single loop with a crossing on it, which divides it into two wings. There are two other crossings on it, such that if you resolve either of the crossings in the other way, you get two loops. This means each of these two crossings must be restricted to one wing, ie, it must go from one wing to itself. If the two crossings are on different wings, then switching the resolutions for both would give you $3$ components.

Thus we have that both crossings are restricted to the same wing. Then, one wing does not have any crossing endpoints on it, which means that we can think of the diagram ignoring the crossing and the empty wing; thus it is not possible for another crossing not the change the number of components, a contradiction. 

\item If $\min(a_i)=1$ and $\max(a_i)=1$. In this case all edge maps are 0, so the square commutes.
\end{enumerate}

This finishes the cases and we have shown commutativity.
\end{proof}

Let $C$ be the cube with one of the crossings not dropped, and $A$ and $B$ be the corresponding cubes when we resolve that crossing in the $0$ and $1$ configurations. 

We will establish maps $A_i \to B_i \to C_i \to A_i$ for $i \in \{0,1\}^{N-1}$ that commute with the maps in the cube, such that $A_i \to B_i \to C_i \to A_i \to B_i$ is exact, and also this splits (I will define what that means in this context).

The maps within $A$, $B$, and $C$ are already defined. As shown in Figure \ref{figure1}, define the maps $B \to C$, $A \to B$, $C\to A$, thus: whenever there are maps between terms with a different number of components, take $m$ or $\Delta$, unless either the source or the target is $0$ (which happens when one of the components of the corresponding unlink has an odd number of endpoints of $\omega$), in which case the map is $0$.

If the source and target unlinks have the same number of components, let the map be $0$ unless an unresolved crossing divides a component into two parts with an odd number of marked points on each side, in exactly one of the source and target, in which case let the map be $\Id$.

\begin{figure}[ht!]
\centering
\includegraphics[width=150mm]{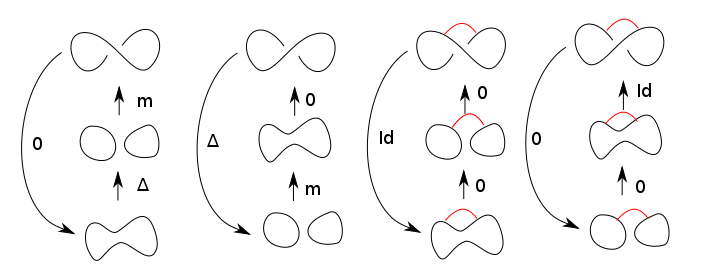}
\caption{}
\label{figure1}
\end{figure}

\begin{lemma} The cube commutes and for $i \in \{0,1\}^N$, the sequences $A_i \to B_i \to C_i \to A_i \to B_i$ are exact.
\label{ABC_exact}
\end{lemma}

\begin{proof} The exactness is easy to see. As for the commutation: we wish to show that squares containing the vertical maps commute. For squares that only involve $A$s and $B$s, we have already shown this above, when we checked that out modified Khovanov homology forms a complex. 

This means it remains to show that the complex:
\begin{diagram}
A_1 & \rTo & A_2\\
\uTo & & \uTo \\
C_1 & \rTo & C_2\\
\uTo & & \uTo \\
B_1 & \rTo & B_2\\
\end{diagram}
commutes.

As before, let us ignore components that are not touched by either of the two crossings in question, and only look at ones that are, ie, the active ones.

Note that two crossings are involved in this picture; the one corresponding to the vertical edges is the one we will be leaving unresolved in row $C$. Call this $x$. The two columns correspond to the resolutions of the other. Call this $y$.

Let $a_1,a_2, a_3, a_4$ denote the number of components in the pseudo-diagram resolutions corresponding to $A_1$, $A_2$, $B_1$, and $B_2$, respectively. Note that $\min(a_i) \leq 2$. Again, we divide into cases:

\begin{enumerate}
\item If $\min(a_i) = 2$. In this case at the corner (of the AB square) with the minimal number of components, neither crossing $x$ nor $y$ goes between components, and since both components have to be involved, that means $x$ is on one component and $y$ on the other. The maps then act independently on corresponding tensor factors, so the squares commute.

\begin{figure}[ht!]
\centering
\includegraphics[width=150mm]{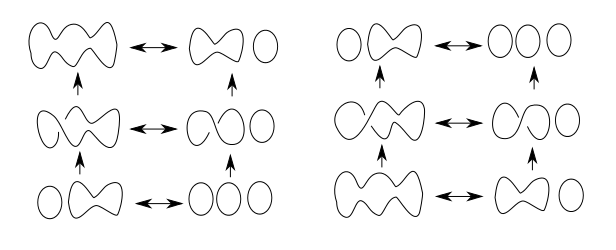}
\caption{}
\label{figure1c}
\end{figure}

\item If $\min(a_i)=1$ and $\max(a_i) = 3$. Then the resolution looks like one of the resolutions in Figure \ref{figure1c} (possibly with $\omega$). In the corner with three components, one of the components does not involve $x$. Therefore it persists in the entire column. If that component has an odd number of endpoints of $\omega$, then any composed map across any square has either vanishing source or vanishing target, so the squares commute.

Otherwise, we have that if the vertical maps on the left are $f$ and $g$, as in the diagram, then the ones on the right are $f \otimes \Id$ and $g \otimes \Id$. Then each square in question is:
\begin{diagram}
\cdot & \pile{\rTo^\Delta \\ \lTo_m} & \cdot\\
\uTo^f & & \uTo^{f \otimes \Id}\\
\cdot & \pile{\rTo^\Delta \\ \lTo_m} & \cdot\\
\end{diagram}
Here, $f$ is one of: $0,\Id, m,\Delta$. If $f = 0$ or $f = \Id$, the diagram clearly commutes. For $f =m$ or $f = \Delta$, the picture is the same as a classical Khovanov diagram where instead of the loop with a crossing as in row $C$, it is just a loop, so the diagram commutes.

\item If $\min(a_i) = 1$ and $\max(a_i) = 2$. In this case if we look at the square formed by $A_1,A_2,B_1,B_2$, ie the one corresponding to the classical Khovanov complex for the two crossings, it must have two vertices with $a_i=1$, opposite each other; the other two vertices have $a_i=2$.

More specifically, $A_1$ and $B_2$ must have one component, and $A_2$ and $B_1$ must have two components; if this were reversed, then the square would represent a link projection for an unlink with two crossings between the components, which cannot be a projection of an alternating link. Thus, the full picture (not including $\omega$) looks as in Figure \ref{figure1d}.

\begin{figure}[ht!]
\centering
\includegraphics[width=150mm]{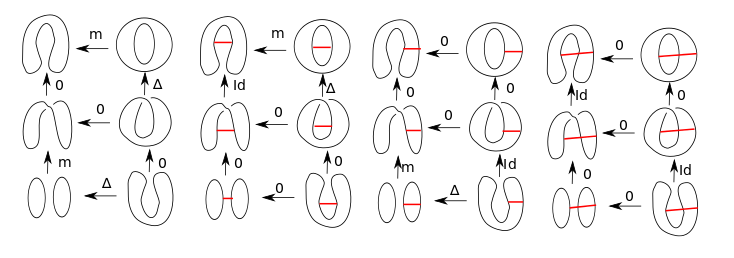}
\caption{In this figure, the top row is the row corresponding to the $A_i$, the middle one to the $C_i$, and the bottom one to the $B_i$. The four cases are based on whether either or both of $A_1$ and $B_2$ are killed by $\omega$}
\label{figure1d}
\end{figure}

Note that in this case $C_1$ and $C_2$ must both have $1$ component, so the map $C_1 \to C_2$ is 0. Thus, to show that the squares commute, it suffices to show that
\[B_1 \to B_2 \to C_2\]
and
\[C_1 \to A_1 \to A_2\]
are zero. These compositions are both either $0$ or $m \circ \Delta$, which is also $0$.
\end{enumerate}

This concludes the proof that the complex composed of $A,B,C$ as exhibited commutes.
\end{proof}

We shall now explain the sense in which the vertical exact sequences ``split'', as we alluded to earlier.

In figure \ref{figure2}, we show the four cases for what the vertical maps could be, only taking into account components with the crossing $x$.

\begin{figure}[ht!]
\centering
\includegraphics[width=150mm]{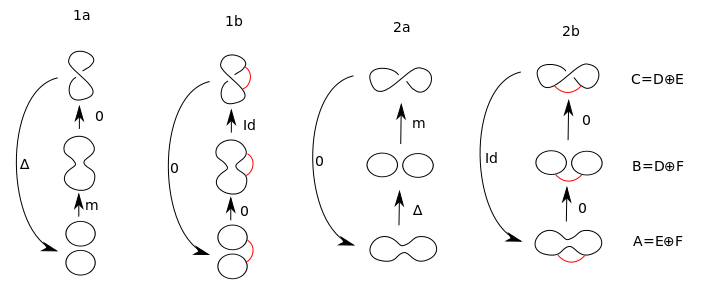}
\caption{}
\label{figure2}
\end{figure}

Let us exhibit $D_i,E_i,F_i$ such that $A_i = E_i \oplus F_i$, $B_i = D_i \oplus F_i$ and $C_i = D_i \oplus E_i$, such that map $A_i \to B_i$ is $0$ on $E_i$ and $\Id:F_i \to F_i$, and similarly for $B_i \to C_i$ and $C_i \to E_i$.

This, again, happens by casework. For the four cases 1a, 1b, 2a, 2b, in figure 2, 

Then $D_i,E_i,F_i$ are given by
\begin{itemize}
\item Case 1a: 
\[A_i = E_i \oplus F_i \text{, with } E_i = (1 \otimes x+x \otimes 1, x \otimes x) \text{, }F_i = (1 \otimes 1, 1 \otimes x)\]
\[B_i = D_i \oplus F_i \text{, with } D_i = 0 \text{, }F_i = (1,x)\]
\[C_i = E_i \oplus D_i \text{, with } D_i = 0 \text{, }E_i = (1,x)\]

\item Case 1b: 
\[A_i = E_i \oplus F_i \text{, with } E_i = 0 \text{, }F_i = 0\]
\[B_i = D_i \oplus F_i \text{, with } D_i = (1,x) \text{, }F_i = 0\]
\[C_i = E_i \oplus D_i \text{, with } D_i = (1,x)  \text{, }E_i = 0\]

\item Case 2a:
\[A_i = E_i \oplus F_i \text{, with } E_i = 0 \text{, }F_i = (1,x)\]
\[B_i = D_i \oplus F_i \text{, with } E_i = (1 \otimes 1, 1 \otimes x)\text{, } F_i = (1 \otimes x+x \otimes 1, x \otimes x) \]
\[C_i = E_i \oplus D_i \text{, with } D_i = (1,x) \text{, }E_i = 0\]

\item Case 2b:
\[A_i = E_i \oplus F_i \text{, with } E_i = (1,x) \text{, }F_i = 0\]
\[B_i = D_i \oplus F_i \text{, with } E_i = 0 \text{, } 0 \]
\[C_i = E_i \oplus D_i \text{, with } D_i = 0 \text{, }E_i = (1,x)\]

\end{itemize}

Note that in the above definition, we have to make some choice with the term $F_i \subset A_i$ in case 1a and $D_i \subset B_i$ in case 2a. Our definition depends on the choice of ordering of the two components. We can do this consistently by choosing some one of the four corners of the crossing we want unresolved and always taking the loop containing that corner to be the first. We call this the \textbf{first loop}.

In defining the $D,E,F$, we are considering the component of the pseudodiagram resolution containing the unresolved crossing. For the rest we tensor up with $A^{\otimes N-1}$ componentwise.

Now we are looking at sequences of cubes $E \oplus F \to D \oplus F \to D \oplus E \to E \oplus F$. We can split $D$ into $\oplus D_i$ where $D_i$ is the direct sum of all parts of $D$ where the sum of the indices is $i$; similarly for $E$ and $F$.
\begin{diagram}
\cdots & \rTo & E_i \oplus F_i & \rTo & E_{i+1} \oplus F_{i+1} & \rTo & \cdots \\
       &      & \uTo           &      & \uTo                   &      &        \\
\cdots & \rTo & D_i \oplus E_i & \rTo & D_{i+1} \oplus E_{i+1} & \rTo & \cdots \\
       &      & \uTo           &      & \uTo                   &      &        \\
\cdots & \rTo & D_i \oplus F_i & \rTo & D_{i+1} \oplus F_{i+1} & \rTo & \cdots \\
       &      & \uTo           &      & \uTo                   &      &        \\
\cdots & \rTo & E_i \oplus F_i & \rTo & E_{i+1} \oplus F_{i+1} & \rTo & \cdots \\
\end{diagram} 
where the horizontal maps are the differentials in the cubes $A,B,C$ and the vertical maps are all identity on one component and $0$ on the other.

Let us consider what the horizontal maps are, ie how the differentials within cubes $A$, $B$, and $C$ look when written in terms of $D,E,F$. Let us consider $d_A:E_i \oplus F_i \to E_{i+1}\oplus F_{i+1}$ and write it as $\mattwo{d_E}{\gamma}{d_{EF}}{d_F}$. Note that $d_{EF} = 0$, because if it were nonzero, consider

\begin{diagram}
D_i \oplus F_i & \rTo^{\mattwo{d_D}{d_{FD}}{d_{DF}}{d_F}} & D_{i+1} \oplus F_{i+1} = B_{i+1}\\ 
\uTo^{0,\Id} & & \uTo^{0,\Id} \\
E_i \oplus F_i & \rTo^{\mattwo{d_E}{\gamma}{d_{EF}}{d_F}} & E_{i+1} \oplus F_{i+1} = A_{i+1}\\
\end{diagram}

Then if $d_{EF} \neq 0$, then some $x \in E_i$ maps nontrivially to $F_{i+1}$, which then maps nontrivially up to $F_{i+1} \subset B_{i+1}$, but this is impossible because the diagram commutes, and $x$ maps to $0 \in D_i \oplus F_i$.

Similarly, $d_{DE}$ and $d_{FD}=0$. Let us remark that for this argument, it is not surprising that we do not need to take into account the choice we made for $F \subset A$ and $D \subset B$, whose definitions required a choice of a corner of unresolved crossing, because the statement that $d_{EF}$, $d_{DE}$, and $d_{FD}$ vanish is saying that the part that the vertical differential kills must map horizontally to the part of the target that the vertical differential kills.

We can now write
\[d_A = \mattwo{d_E}{\gamma}{0}{d_F},\]
\[d_B = \mattwo{d_D}{0}{\beta}{d_F},\]
and 
\[d_C = \mattwo{d_D}{\alpha}{0}{d_E}\]

\begin{lemma} In the notation above, $\alpha = 0$, where $\alpha:E \to D$.
\label{alpha0}
\end{lemma}

\begin{proof} Notice that in the cube $C$, in each vertex of the cube we have $D_i=0$ or $E_i=0$ (notice this by checking all the cases; see figure \ref{figure2}). The cases with nonzero $E$ are 1a and 2b, and the ones with non-zero $D$ are 1b and 2a. Hence, a part of the differential on $C$ with $\alpha \neq 0$ must be
\[(1a \text{ or } 2b) \to (1b \text{ or } 2a)\]

We divide into cases:
\begin{itemize}
\item $1a\to 1b$: Looking at the square on the $A$ and $B$ rows,
\begin{diagram}
B_1 & \rTo & B_2\\
\uTo & & \uTo \\
A_1 & \rTo & A_2\\
\end{diagram}
the vertical maps come from merges, and the horizontal map on row $A$ must come from a split, because it goes from a resolution that doesn't split $\omega$ to one that splits $\omega$. Counting numbers of components, we see that the horizontal map on level $B$ must also be a split. Thus, if $B_1$ has $i$ components, then $A_1,B_2$ have $i+1$ and $A_2$ has $i+2$.

Counting only active components (components that involve at least one of the two crossings), $i$ is $1$ or $2$. In case $i=1$, the original link cannot be alternating. If $i=2$, we are looking at two crossings on two separate components in the $B_1$ corner, at least one of which has $\omega$ split across the crossing, which means either $A_1=0$, in which case the first column is not in case $1a$, or $B_2=0$, in which case the second is not in case $1b$, a contradiction.

\item $2b \to 2a$, This proof works exactly the same way; the vertical maps come from splits, and for the horizontal map $B_1 \to B_2$ to go from splitting $\omega$ to not splitting $\omega$, it must be a merge, and then the same argument applies.

\item $1a \to 2a$ and $2b \to 1b$. Note that the number of components in the resolution for $A_1$ must be of opposite parity to the number of components in the resolution for $A_2$, but this means that the number of components in the resolution for $C_1$ must be the same as that for $C_2$. Thus for this map, the differential $C_1 \to C_2$ is simply zero, because on level $C$ we defined the differential to be zero whenever it goes between two resolutions of the same number of components.

Thus $\alpha =0$. 
\end{itemize}
\end{proof}

From the above analysis, we see that the cone of the differential $A \to B$ is $ \oplus D_{i-1}  \oplus F_{i-1} \oplus F_i \oplus E_i$, with the differential 
\[\cdots \to D_{i-1}  \oplus F_{i-1} \oplus E_i \oplus F_i \to D_{i}  \oplus F_{i} \oplus E_{i+1} \oplus F_{i+1} \to \cdots\]
given by
\[{\left[\begin{array}{cccc} d_D  & 0 & 0 & 0\\ \beta & d_F & 0 & \Id \\ 0 & 0 & d_E & \gamma \\ 0 & 0 & 0 & d_F\\ \end{array} \right]}\]

Observe, moreover, that $D$, $E$, and $F$ are chain complexes and $\beta$ and $\gamma$ are chain maps: This is because 
\[0=d_B^2 = \mattwo{d_D}{0}{\beta}{d_F}^2 = \mattwo{d_D^2}{0}{\beta d_D+d_F \beta}{d_F^2}\]
so $d_D^2=0$, $d_F^2=0$, and $\beta$ is a chain map. Using $d_A$, we can show the same for $d_E$ and $\gamma$.

Let us consider the composition $\gamma \circ \beta:D_i \to E_{i+2}$.

\begin{lemma} The cone of the map $A \to B$, that is the complex $ \cdots \to \oplus D_{i-1}  \oplus F_{i-1} \oplus F_i \oplus E_i \to \cdots $
is quasi isomorphic to the complex
\[\cdots \to D_{i-1}   \oplus E_i \to D_{i}  \oplus E_{i+1} \to \cdots \]
with differential $d=\mattwo{d_D}{0}{\gamma \circ \beta}{d_E}$, which we can think of as a complex based on the psuedodiagram that gave rise to $C$, but with an extra differential term $\gamma \beta$ which bumps the degree up by $2$ (ie, the differential now has components in (un-adjusted for self-intersection of cobordisms) cohomological dimension both $1$ and $2$). 
\end{lemma}

\begin{proof} Consider the maps
\[\pi: D_{i-1} \oplus F_{i-1} \oplus E_i \oplus F_i  \leftrightarrow D_{i-1}   \oplus E_i:i \]
given by 
\[i(x,z) = (x,0,z,\beta x) \text{, and }\pi(x,y,z,w) = (x, z+\gamma y).\]

Let us check that these are chain maps:
\[d(i(x,z)) = {\left[\begin{array}{cccc} d_D  & 0 & 0 & 0\\ \beta & d_F & 0 & \Id \\ 0 & 0 & d_E & \gamma \\ 0 & 0 & 0 & d_F\\ \end{array} \right]} 
{\left[\begin{array}{c} x \\ 0 \\ z \\ \beta x\\ \end{array} \right]} = {\left[\begin{array}{c} d_Dx \\ \beta x+ \beta x \\ d_E z + \gamma \beta x \\ d_F \beta x\\ \end{array} \right]} \oset{\heartsuit}= {\left[\begin{array}{c} d_Dx \\0 \\ d_E z + \gamma \beta x \\ \beta d_D x\\ \end{array} \right]}\]
\[ =i \vecttwo{d_D x}{\gamma \beta x+d_Ez}= i(d(x,z)) \]
where the equality $\heartsuit$ comes from the facts that $\beta$ is a chain map and that we are working over characteristic 2.

For the other direction, we have
\[d(\pi(x,y,z,w)) = \mattwo{d_D}{0}{\gamma \beta}{d_E} \vecttwo{x}{z+\gamma y} = \vecttwo{d_Dx}{\gamma \beta x+d_E z+d_E \gamma y}\]
\[=\vecttwo{d_Dx}{d_E z+ \gamma w + \gamma \beta x + \gamma d_F y + \gamma w} =  \pi {\left[\begin{array}{c} d_Dx \\ \beta x+ d_F y + w \\ d_Ez+\gamma w \\ d_F w\\ \end{array} \right]}  =  \pi(d(x,y,z,w)) \]

Having checked that the maps are chain maps, we proceed to show that their compositions are chain homotopic to the identity.

For $\pi \circ i$, note that $\pi \circ i(x,z) = \pi(x, 0, z, \beta x) = (x, z+ \gamma 0)  = (x,z)$, so $\pi$ is a true left inverse of $i$.

For $i \circ \pi$, we have 
\[i \circ \pi (x,y,z,w) = i(x, z+\gamma y)= (x, 0, z+\gamma y, \beta x).\]
To show that this is chain homotopic to the identity, it suffices to show that $(0,y, \gamma y, \beta x+w)$ is chain homotopic to $0$. 

Consider 
\[h ={\left[\begin{array}{cccc} 0  & 0 & 0 & 0\\ 0 & 0 & 0 & 0 \\ 0 & 0 & 0 & 0 \\ 0 & 1 & 0 & 0\\ \end{array} \right]}\]
Then we have
\[(hd+dh)(x,y,z,w) = \left(
                  \begin{array}{cccc}
                   0 & 0 & 0 & 0 \\
                   0 & I & 0 & 0 \\
                   0 & \gamma & 0 & 0 \\
                   \beta & 0 & 0 & 1 \\
                  \end{array}
                  \right)
\left(\begin{array}{c} x \\ y \\ z \\ w \\ \end{array}\right)
 =  \left(\begin{array}{c} 0 \\ y \\ \gamma y \\ \beta x + w \\ \end{array}\right)\]
As desired.

\end{proof}

Note that terms of the form $\gamma \circ \beta: C_i \to C_{i+2}$ will involve 2 crossings (other than the dropped one), so we can think of it as a diagonal map on a square in $C$.

We can now show that $\gamma \circ \beta =0$ for alternating links.

\begin{lemma} For alternating link projections, with our notation as above, $\gamma \circ \beta =0$. Thus, for alternating link projections, the complex $A \to B$ is quasi-isomorphic to the complex $C$ with differential $\mattwo{d_D}{0}{\gamma \circ \beta}{d_E}$, but as $\gamma \circ \beta=0$, this is $\mattwo{d_D}{0}{0}{d_E}$, which is the complex for one dropped crossing, by Lemma \ref{alpha0}.
\end{lemma}
\begin{proof} We will exhaustively go through all cases where $\beta \neq 0$ or $\gamma \neq 0$ to see what $\gamma \circ \beta$ could be. To understand $\beta$ and $\gamma$, we only need to look at two crossings, the one that is unresolved in $C$, which we call $x$, and the other one, which we call $y$.

Consider the types 1a, 1b, 2a, and 2b, as in Figure \ref{figure2}.

Note that $\beta$ can only be nonzero if we are going from a type where $D \neq 0$ to a type where $F \neq 0$. So this is only possible when we are going:
\[\beta \neq 0 \text{:  (1b or 2a) $\to$ (1a or 2a)}.\]
Similarly,
\[\gamma \neq 0 \text{:  (1a or 2a) $\to$ (1a or 2b)}.\]

We can eliminate some cases:
\begin{itemize}
\item Let's consider a map $1b \to 1a$. If we start with the 1b picture and switch crossing $y$, then on level $A$, we go from splitting $\omega$ to not splitting $\omega$. Thus the map $A_1 \to A_2$ is a merge. However, $A_1 \to B_1$ and $A_2 \to B_2$ are both merges, by the definition of types $1a$ and $1b$. So all four maps in the $AB$ square are merges.

This, however, is impossible, because the fact that we are going from $1b$ to $1a$ means that both the maps $A_1 \to A_2$ and $A_1 \to B_1$ must come from merges between the two components which have an odd number of endpoints of $\omega$ in $A_1$, but for two merges between the same two components, the two crossing diagram this is resolving has to be a Hopf link, which means the maps $A_2 \to B_2$ and $B_1 \to B_2$ must be splits.

\item Similarly, for maps $2a \to 2b$, by counting components, the horizontal arrows must either both arise from merges, or both arise from splits, but the condition that only the $B_2$ diagram separates $\omega$ contradicts this.
\end{itemize}

Thus for $\gamma \circ \beta \neq 0$, the only possibilities are:
\begin{itemize}
\item $1b \to 2a \to 1a$
\item $2a \to 1a \to 1a$
\item $2a \to 1a \to 2b$
\item $2a \to 2a \to 1a$
\end{itemize}

All four of these possibilities involve some map $2a \to 1a$. Again, let us consider the $AB$ square for this:
\begin{diagram}
B_1 & \rTo & B_2\\
\uTo & & \uTo \\
A_1 & \rTo & A_2\\
\end{diagram}
Since this picture is $2a$ to $1a$, the vertical maps are $\Delta$ on the left and $m$ on the right. Thus, counting components, we see that the horizontal maps can only be $\Delta$ on the bottom and $m$ on the top. The map $C_1 \to C_2$ therefore comes from a resolution switch that doesn't change the number of components. The crossing therefore has to go between the wings of a loop with $x$ on it. Thus, if we leave both $x$ and $y$ unresolved, we get either a Hopf link or a two component unlink with two crossings between components. The $00$ resolution of this picture has only one component, so of these two possibilities, it must be the unlink. Such an unlink projection, however, is not alternating, so we have reached a contradiction.
\end{proof}

\begin{cor} For alternating link projections $P_L$ with $\omega$, the rank of the Khovanov homology with $\omega$ we defined is invariant with respect to dragging endpoints of $\omega$ around a component.
\label{indep_slide_omega}
\end{cor}
\begin{proof} The only problem case is when you drag an endpoint diagonally across a crossing, and in this case, we can compare both sides to the complex with that crossing dropped. The one crossing dropped complex doesn't see where the endpoint is.
\end{proof}

\begin{cor} For alternating link projections $P_L$ with $\omega$, where $L$ is not split, the rank of the marked link Khovanov homology of $P_L$ with $\omega$ does not depend on $\omega$.
\label{indep_omega}
\end{cor}
\begin{proof} Consider some arc in $\omega$ with its two endpoints on two components of $L$. We may call them $L_1,L_k$, where there are components $L_1,L_2,\ldots L_k$ such that $L_i$ and $L_{i+1}$ have a crossing between them (because $L$ is not split). Then we can replace $\omega$, with $\cup \omega_i$, where $\omega_i$ has one endpoint on $L_i$ and one on $L_{i+1}$ and is sitting on adjacent branches of a crossing between $L_i$ and $L_{i+1}$, which we call $X_i$.

Then adding each $\omega_i$ does not affect the complex, because we can consider the complex with $X_i$ unresolved. In the resulting pseudo-diagram resolution, $\omega_i$ cannot kill components (because $\omega_i$ sits across crossing $X_i$, which is not being resolved), meaning $\omega_i$ does not affect the underlying groups in the complex with $X_i$ unresolved, and it is easy to see from the definition that it also does not affect any of the differentials. Thus, it does not affect the homology of the complex with $X_i$ unresolved, which means it also does not affect the homology of the original complex.

Now we can remove the $\omega_i$ one by one, not changing the homology of the complex, so removing the entire arc does not change the homology of the complex. We may further remove all the arcs of $\omega$ one by one, as desired.

\end{proof}

As a consequence of this last corollary, we see that the marked Khovanov homology for alternating link projections we defined is just the usual Khovanov homology of the link. In particular, it is also a link invariant, completing the proof of Theorem \ref{naive_khov_invar} for $H(P_L,\partial \omega)$.

For the reduced version, note that the same proof holds: We may still form the complex with the dropped crossing, by taking quotients by $\la x \ra$ appropriately. Lemma \ref{ABC_exact} still holds because the sequence $A/\la x \ra \to A \otimes A/\la x \ra \to A \to A/\la x \ra \to A/\la x \ra$ with maps given by $\Delta_\red$, $m_\red$, and $0$ is still exact.

Then in the splitting of the complexes $A$, $B$, and $C$ into $D$, $E$, and $F$ there is a little bit of subtlety with the choice of $F$ in case $1a$ and $E$ in case $1b$. In particular, let us make the choice so that if the the base point is on the active components, then it is on the second component, so that, in case $1a$, we have $F_i = (1 \otimes 1)$ and $E_i = (x \otimes 1)$, and for case $1b$, $E_i = (1 \otimes 1)$ and $F_i = (x \otimes 1)$. Then, we again get direct sum decompositions.

This change in the choice for $F_i$ and $E_i$ does not affect the proof that $d_{EF}, d_{FD}$, and $d_{DE}$ vanish in general, nor for $\alpha$ or $\gamma \beta$ for alternating links.

Thus, by the same argument, we get that $H_\red(P_L,\partial \omega)$ does not depend on $\omega$, completing the proof of Theorem \ref{naive_khov_invar}.


\subsection{A discussion of filtrations for non-alternating links}\label{q_section}

\subsubsection{Dropped crossings without $\omega$}\label{subsubsec_drop_cross_no_omega}

We have shown that for alternating link projections, $\omega$ has no effect on $H(P_L,\partial \omega)$, by way of a complex that comes from dropping one crossing. The latter complex was inspired by the crossing dropping procedure that Kronheimer and Mrowka introduced in \cite{KM_filtrations}. 

Let us omit $\omega$ for the moment and examine more carefully how the property that the projection was alternating came into our picture, and how it relates to the one in \cite{KM_filtrations}. This will give another explanation for why one can drop a crossing when computing the Khovanov homology of alternating link projections (without $\omega$).

In link projection that is not necessarily alternating, most of the statements in subsection \ref{alternating_marked_points_section} regarding Khovanov complexes computed with dropped crossings continue to hold, though there are more cases to consider, and we must take more care when defining the differentials in the pseudo-diagram in the case of maps between resolutions with the same number of components.

In particular, this more subtle complex still commutes as in Lemma \ref{dropped_commutes}, and still fits into a larger complex with the dropped crossing resolved as in Lemma \ref{ABC_exact}. Moreover, the exact sequence still has the splitting into $D,E,F$, with $\alpha=0$. The main difference is that now $\gamma \beta$ does not necessarily vanish.

Thus, we still get a complex based on the resolutions with one dropped crossing, but now the cube may have diagonal maps across squares.

Let us compare this to what happens in \cite{KM_filtrations}, in which Kronheimer and Mrowka consider an oriented link projection $P_L$ and a subset $N$ of its crossings, such that resolutions of $P_L$ at the $N$ crossings yield pseudo-diagram resolutions. They form the complex $(\oplus C(L_v),d^\#)$ where $L_v$ runs over the resolutions of $L$, and chain maps arise from counting solutions to the ASD equation on the corresponding cobordisms. They further consider two filtrations on the complex, $h$ and $q$, coming from the topologies of the cobordisms, and show that the isomorphism types of the pages of the corresponding spectral sequences for both of these filtrations are invariants of $L$.

These filtrations are given by:
\[q = Q - \(\sum_{c \in N} v(c)\) + \frac{3}{2} \sigma(v, o) -n_++2n_-\]
and
\[h = - \(\sum_{c \in N} v(c)\) + \frac{1}{2} \sigma(v, o) +n_-,\]
where $Q$ is a grading on $(R[x]/x^2)^{\otimes l}$, which has $x$ in grading $-1$ and $1$ in grading $1$; on the summand $C_v$, where $v(c)$ denotes the resolution of $c$, that is, it is $0$ or $1$, depending on how $C$ is resolved in the pseudo-diagram resolution at $C_v$, with $v(c)=1$ for the $0$ resolution and $v(c)=0$ for the $1$-resolution; $o$ is a chosen vertex of the cube where the corresponding resolution can be oriented in a way that is consistent with the orientation on $L$; $\sigma(v,u)$ is the self intersection $S_{vo} \cdot S_{vo}$ of the cobordism $S_{vo}$ when $u \geq v$, and is defined for $u \not \geq v$ in such a way that it is additive, that is for $u,v,w$, $\sigma(w,v)+\sigma(v,u) = \sigma(w,u)$; and $n_+$ and $n_-$ the number of positive and negative crossings of the $N$ crossings, respectively.

\begin{rmk} In this subsection, we will be following the notation in \cite{KM_filtrations} and \cite{KM_unknot}, in which the maps go from the $1$ resolution to the $0$ resolution, so we are actually looking at the resolutions of the mirror image of the link. This is the reverse of the direction our maps were going in subsection \ref{alternating_marked_points_section}. 
\label{v_geq_u_notation}
\end{rmk}

Consider the first page of this spectral sequence, $(E_1,d_{1,h})$, where $d_{1,h}:F_i/F_{i+1} \to F_{i+1}/F_{i+2}$ is the map induced by the differential on the page of the spectral sequence arising from the $h$ filtration. Kronheimer and Mrowka show in \cite{KM_filtrations} Proposition 10.2 that when all the crossings of $L$ are resolved, only maps along the edges of the cubes come into $d_{1,h}$, and they show in \cite{KM_unknot} Section 8.2 that these maps agree with the Khovanov edge maps $d_1$.

Let us consider what happens when one crossing is left unresolved. In this case, the edge cobordisms in question have two possibilities: If the edge corresponds to a change in number of components in the resolution, then the cobordism is a pair of pants, and otherwise it is a twice punctured $\rr P^2$. (Here we are only concerning ourselves with the parts of the cobordisms between the active components; the rest of the cobordisms consist of cylinders, which contribute to neither $\chi(S)$ nor $S \cdot S$).

If a cobordism $S$ is a pair of pants, then $S \cdot S = 0$. If $S$ is a twice punctured $\rr P^2$, consider the two crossing projection given by the unresolved crossing and the crossing that corresponds to the edge in question. This cobordism is as depicted in Figure \ref{unlink_hopf_cobords}, with the map left to right corresponding to the Hopf link and right to left corresponding to the unlink. (This is the opposite to the Khovanov differentials because the $E_1,d_1$ page of the instanton complex corresponds to the mirror image of the Khovanov complex.)

\begin{figure}[ht!]
\centering
\includegraphics[width=45mm]{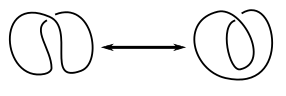}
\caption{}
\label{unlink_hopf_cobords}
\end{figure}

\begin{figure}[ht!]
\centering
\includegraphics[width=45mm]{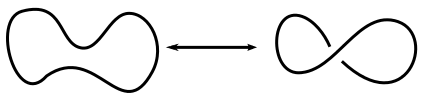}
\caption{}
\label{cobord_post_sliding}
\end{figure}

By sliding the arcs around it is easy to see that these are the same as the cobordisms in Figure \ref{cobord_post_sliding}. It was shown in Lemma 7.2 of \cite{KM_unknot} that the cobordism going from the right to the left in \ref{cobord_post_sliding} has self intersection $+2$, so the one left to right has self intersection $-2$. We conclude that for cobordisms corresponding to an unlink the self intersection is $+2$, and for Hopf links, it is $-2$. 

If $L$ is an alternating link projection, for $v \geq u$, if any punctured $\rr P^2$s are involved in the path from $v$ to $u$, they must correspond to Hopf links, rather than unlinks, because if we resolve some crossings of an alternating projection, the resulting projection is still alternating. Thus, for alternating projections, $S_{vu} \cdot S_{vu} \leq 0$. Consequently, the difference in $h$ satisfies
\[h(u)-h(v) = \sum v- \sum u -\frac{1}{2}\sigma(v,u),\]
which is at least $1$ for edges and at least $2$ for diagonals in the cube. Therefore the $(E_1,d_1)$ page does not involve diagonal maps in the cube, and it is easy to check that it agrees with the cube from our previous section.

If $L$ is non-alternating, however, there may be diagonal maps on the cube that are part of $d_1$, that is, which shift $h$ by $1$. This is because the change in $h$ as you go along the diagonal is equal to the change in naive grading (the grading on the cube), shifted by $\frac{1}{2}\sigma(v,u)$, but now for $v \geq u$, $\sigma(v,u)$ could be positive. Thus some diagonals could change $h$ by only $1$.

Let us consider which diagonals can appear in the $d_{1,h}$ level. It was shown in \cite{KM_filtrations} that when only one crossing is dropped, $|\sigma(v,u)| \leq 2$, for any two vertices, so the change in $h$ can be at most $1$ off from the change in naive grading. Thus the $d_1$ includes only edge maps and diagonal maps across squares. 

Now, using the $q$ grading, one can write down what the diagonal maps across squares must be.

\subsubsection{Figure \ref{non_alt_eg}: an example non-alternating link projection with $\omega$}

In the case of links with marking data, the pages with respect to the $q$ filtration in \cite{KM_filtrations} no longer provide invariants. This makes sense because the $q$ filtration comes from studying the maps in the cube of instanton complexes, which come from counting points in zero dimensional moduli spaces of certain anti-self-dual connections.

More specifically, for a cobordism $S \subset S^3 \times \rr$ from $L_v \subset S^3$ to $L_u \subset S^3$, and singular bundle data $P_\Delta$ on $(S, S^3 \times \rr)$, we are considering connections on $P_\Delta$ that satisfy the perturbed ASD equation and agree with $\beta_1$ and $\beta_0$ on the ends. For such connections, the action, which is given by
\[\kappa(A) =\frac{1}{8 \pi^2} \int_{X-\Sigma} \tr(F_A \wedge F_A), \]
is a homotopy invariant of the path $A$
and also satisfies
\[\kappa(A) \equiv  -\frac{1}{4} p_1(P_\Delta)[X_\Delta] +\frac{1}{16} S \cdot S \pmod{1/2},\]
Where $X_\Delta$ and $P_\Delta$ are set up as follows. 

Let $X=S^3 \times \rr$, and let $S \subset X$ be a two dimensional submanifold. Recall from \cite{KM_unknot} that a $PU(2)$ bundle $P$ on $X\backslash S$ modelled on $\frac{1}{4} i d\theta$ gives rise to a double cover $S_\Delta$ of $S$ coming from the two ways to extend $P$ to $S$. From this, Kronheimer and Mrowka constructed a non-Hausdorff space, $X_{\Delta}$, equipped with a map $X_{\Delta} \to X$ that is an isomorphism over $X\backslash S$ and such that the pre-image of $S$ is $S_\Delta$, and a $PU(2)$ bundle $P_\Delta$ over $X_\Delta$, which agrees with $P$ outside of a neighbourhood of $S_\Delta \subset X_\Delta$.

In \cite{KM_unknot} Kronheimer and Mrowka further constructed a space $X_\Delta^h$, a Hausdorff space with the same weak homotopy type as $X_\Delta$, and showed that $[X_\Delta]$ is a half integral class in $H_4(X_\Delta^h;\qq)$. Thus, for $p_1(P_\delta) \in H^4(X_\Delta^h,\zz)$, we may consider the half integer $p_1(P_\Delta)[X_\Delta]$. Moreover, since $p_1(P_\Delta) \equiv \cP(w_2(P_\Delta)) \pmod{4}$, where $\cP$ is the Pontryagin square, $p_1(P_\Delta)[X_\Delta] \equiv \cP(w_2(P_\Delta))[X_\Delta] \pmod{2}$.

For $\beta_1$ and $\beta_0$ solutions to the perturbed Chern Simons functional on the ends (flat connections in the unperturbed case), the dimension of the moduli space of solutions to the perturbed ASD equation that agree with $\beta_1$ and $\beta_0$ on the ends, in a homotopy class of paths with action $\kappa$ is given by the formula
\[\dim(M_\kappa(S; \beta_1, \beta_0)) = 8 \kappa + \chi(S)+\frac{1}{2} S \cdot S + Q(\beta_1)-Q(\beta_0)+ \dim(G)\]
where $Q$ is the grading on $A^\otimes n$ defined in subsection \ref{subsubsec_drop_cross_no_omega}, $G$ is the space of metrics over which the moduli space of connections sits, and $\chi$ is the Euler characteristic. The coefficient of $[\beta_0]$ in the image of $[\beta_1]$ under the differential is then given by counting the number of points in the zero-dimensional moduli space, ie, those paths with
\[8 \kappa + \chi(S)+\frac{1}{2} S \cdot S + Q(\beta_1)-Q(\beta_0)+ \dim(G)=0.\]

The non-negativity of the action for anti-self-dual connections implies that for small perturbations, $\kappa \geq 0$. Moreover, in the situation without $\omega$, $\frac{1}{4} p_1(P_\Delta)[X_\Delta] \equiv 0 \pmod {1/2}$, because $p_1(P_\Delta)$ is a multiple of $4$ and $[X_\Delta]$ is a half integral class. Thus, $\kappa \geq \frac{1}{2} S \cdot S - 4 \lfloor \frac{S \cdot S}{8} \rfloor$.

The proof of invariance of the isomorphism type of the complex in the category of homotopy classes of $q$ or $h$ filtered chain complexes comes from keeping track of constraints on the edges and diagonal maps in the cube coming from the dimension formula above. 

In the fully resolved case, invariance of the isomorphism type of the Khovanov homology could be extracted from looking at the $h$ filtration for the diagonal maps and showing that there are no diagonal maps on the cube with $h$-order $1$. Thus, the isomorphism type of the Khovanov homology agrees with that of the $E^2$ page of the instanton complex with respect to the $h$ filtration, and is therefore a link invariant.

In the case of the counterexample in Figure \ref{non_alt_eg} from subsection \ref{alternating_marked_points_section}, we can still use the dimension calculation essentially to write down what the diagonal maps on the cube of instanton complexes have to be. Consider the cube of resolutions in Figure \ref{trefail_omega_resn}. The groups $C_{1,0,0}$, $C_{1,1,0}$, $C_{1,0,1}$, and $C_{0,1,1}$ vanish. 

\begin{figure}[ht!]
\centering
\includegraphics[width=120mm]{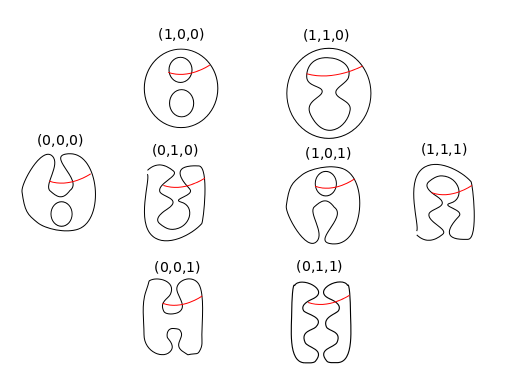}
\caption{}
\label{trefail_omega_resn}
\end{figure}

The maps $C_{0,0,0} \to C_{0,1,0}$ and $C_{0,0,0} \to C_{0,0,1}$ can be seen to be merge maps, as seen in Section 8 of \cite{KM_unknot}. The remaining possible maps are $C_{0,0,0} \to C_{1,1,1}$, $C_{0,1,0} \to C_{1,1,1}$, and $C_{0,0,1} \to C_{1,1,1}$.

From the definition given in \cite{KM_unknot}, $\dim(G)$ is one less than the number of crossings in the cobordism. It is then easy to see that $\chi(S) = -\dim(G)-1$, so $\chi(S)+\dim(G) = -1$. Moreover, since the cobordisms in the diagram are orientable, $S \cdot S = 0$. Thus for moduli spaces of dimension $0$, we must have
\[Q(\beta_0)-Q(\beta_1) = 8 \kappa -1.\]

As in \cite{KM_filtrations}, we have $\kappa \geq 0$ and $\kappa = \frac{1}{4} p_1(P_\Delta)[X_\Delta]$. However, $p_1(P_\Delta)$ is no longer a multiple of 4. To figure out what it is instead, let us consider the cobordisms in question. The cobordisms $(0,1,0) \to (1,1,1)$ and $(0,0,1) \to (1,1,1)$ are twice punctured tori and the cobordism $(0,0,0) \to (1,1,1)$ is a thrice punctured torus. If we cap off the ends, we get a torus, with $\partial\omega$ given by a circle that winds once around each representative of $H^1$.

To calculate the action in this situation, let us consider the double branched cover. The double branched cover of $T^2$ in $S^4$ is $S^2 \times S^2$, with $T^2 = S^1 \times S^1$ sitting inside it as the product of the equators of the two $S^2$s. In this picture, $\omega \subset S^4$ represents the $H^2(S^2 \times S^2)$ class $[S^2 \times \text{pt}]+ [ \text{pt} \times S^2]$. Consequently, on the double branched cover, $\mathcal{P}(w_2(P)) \equiv 2 \pmod{4}$, where $\mathcal{P}$ denotes the Pontrjagin square. Thus, $p_1(P) \equiv 2 \pmod{4}$, and $\kappa = -\frac{1}{4} p_1(P)[S^2 \times S^2] \equiv \frac{1}{2} \pmod{1}$. The action on the double branched cover is twice the action on the base, so on the original space, $\kappa \equiv \frac{1}{4} \pmod{2}$. 

From here, we see that $Q(\beta_0)-Q(\beta_1) = 8 \kappa -1 \equiv 1 \pmod 4$, so, by parity, the only possible diagonal map is the one $C_{0,0,0} \to C_{1,1,1}$, which takes $x \otimes x$ to $x$ and either $1 \otimes x$ or $x \otimes 1$ to $1$. We know these maps must appear in the instanton complex, because otherwise it would be impossible to end up with the right value for the instanton homology.

\begin{rmk}This does not tell us which of the maps $1 \otimes x \to 1$ and $x \otimes 1 \to 1$ happens. The specific map in the chain complex may depend on the choice of perturbation. 
\end{rmk}


\subsection{Modifying the $q$ filtration in the presence of $\omega$}

In the previous section, we explained what happens to Kronheimer and Mrowka's $q$ and $h$ filtrations when a crossing is dropped in the case of non-alternating links. In this section we give a modification of the $q$ filtration to show an analogous result to the $q$ part of Corollary 1.3 in \cite{KM_filtrations}, which stated that the isomorphism types of the pages of the spectral sequence with respect to the $q$ filtration are link invariants.

For a projection $P_L$ of a link, taking the cube of pseudo-diagram resolutions with $0$, $1$, or $2$ adjacent (meaning there are no crossings or endpoints of $\omega$ between them), opposite sign dropped crossings, with certain $\omega$, we define a modified version of the $q$ filtration for the cube. Let us define the particular kind of $\omega$ that we would like to work with. 

\begin{defn} We say that $\omega$ is ``trivial'' at a resolution if no component has an odd number of end-points. We say that $\omega$ is ``good'' if for every cobordism between projections with trivial $\omega$ (ie, for all diagonals $vu$ of the cube, including those which do not satisfy $v \geq u$), once we cap off the ends, and consider the resulting closed orientable surface with genus, $\omega$ does not intertwine any of the genus. This is equivalent to saying that for every such cobordism, $\cP(w_2) \equiv 0 \pmod{4}$, where $\cP$ is the Pontrjagin square.
\end{defn}

In this subsection we will show the following theorem regarding good $\omega$.

\begin{thm} Let $P_L$ be a link projection in $S^2$ and $\omega$ be good marking data. Let $B \subset S^2$ be a ball containing $\omega$. Then the isomorphism type of $Kh(P_L,\omega)$ as defined in the introduction is a tangle invariant of $P_L \cap (S^2 \backslash B)$. 
\label{e1_tangle_invar}
\end{thm}

Similarly to the proof of the main theorem in \cite{KM_filtrations}, we will accomplish this by way of a filtration on the instanton complex. Before introducing the filtration, it will be useful to give a property of good $\omega$.

\begin{lemma} For good $\omega$, in the (fully resolved) cube of resolutions, if $u,v$ are two vertices such that $\omega$ is trivial at both of these vertices, then there is a path from $v$ to $u$ that only goes through vertices at which $\omega$ is trivial. Moreover, there is such a path of length $|v-u|_1$ and for $v \geq u$, there is such a path $v= v_1 \geq v_2 \geq \cdots \geq v_k =u$. 
\end{lemma}
\begin{proof} Consider going from the resolution of $v$ to the resolution of $u$ applying the following steps greedily:
\begin{enumerate}
\item merge
\item split into pieces with a remaining crossing between them; ie that will later be merged
\item split into pieces with no remaining crossing between them
\end{enumerate}
Then, it is easy to see that the sequence of moves must be of the form:
\[mm \ldots m (\Delta m) (\Delta m) \ldots (\Delta m) \Delta \Delta \ldots \Delta\]
because after we do merges until we cannot do any more, we have components with only crossings to themselves. Then if we do a type 2 split, we immediately do a merge, and we are again in a situation where no crossings go between components. This proceeds until we can no longer do 1 or 2, at which point there are only splits left.

Now note that since we started and ended with trivial $\omega$, the initial merges and final splits all preserve this trivialness. For the $(\Delta m)$s going from trivial $\omega$ configurations to trivial $\omega$ configurations, if we look at the cobordism capped off, it is a torus, and it is easy to see that if $\omega$ does not intertwine this torus, there must be at least one crossing we can split at that does not make $\omega$ non-trivial, so we split at that crossing. This completes the proof.
\end{proof}

Let us now define a filtration $q$ on the complex associated to the cube of pseudo-diagram resolutions. We will consider in particular three types of cubes of pseudo-diagram resolutions: those that come from a projection for which we resolve all crossings, those for which we resolve all but one crossing, and those for which we resolve all but two adjacent, opposite-sign crossings.

Our $q$ is only defined for vertices of the cube at which $\omega$ is trivial; at other vertices, the group is $0$ anyway, so it does not matter what filtration we choose). 

For a generator $q$ corresponding to a critical point for the resolution at a vertex $v$, define
\begin{equation}
q(a) = Q(a)- \sum_c v(c) + \frac{3}{2} \sigma(v, o)+ \pi(v,o),
\label{q_defn}
\end{equation}
where $\pi(v,u)$ is defined below, $o$ is a globally chosen vertex of the cube so that the $\omega$ is trivial at that vertex. 

\begin{rmk} There is a choice of $o$ involved in the definition of $q$, but this will not matter, because the results we will extract from the $q$ filtration will only require $q$ to be defined up to a constant shift for the whole complex.
\end{rmk}

In this subsection, as in the previous, our maps are going from $v$ to $u$ with $v\geq u$ (see Remark \ref{v_geq_u_notation}). 

\begin{defn} Let $D$ be the set of dropped crossings. Let $\pi(v,u) = \pi(u)-\pi(v)$, where \[\pi(v) = \sum_{c \in D} (-2)(\text{sign}(c))s_\omega(c)\]
where $\text{sign}(c) = \pm 1$ and $s_\omega(c)$ is $1$ or $0$ depending on the parity of the number of $\omega$ endpoints on each of the wings that $c$ divides its component into, if applicable; that is
\[s_\omega(c) = \begin{cases} 0 & \text{$c$ does not divide one component into two components}\\
0 & \text{each wing has an even number of endpoints}\\
1 & \text{each wing has an odd number of endpoints}\\
\end{cases}\]
where the first of the three cases is only possible when there are two dropped crossings and the picture looks like the middle picture in Figure \ref{drop2_maps}. In particular, when there are no dropped crossings, $\pi=0$. 

\begin{figure}[ht!]
\centering
\includegraphics[width=120mm]{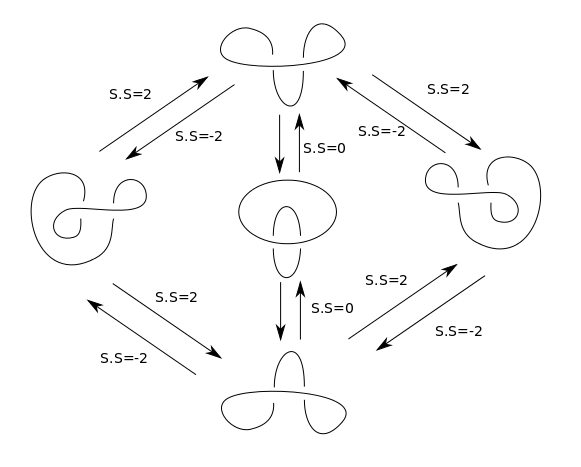}
\caption{This figure shows all possible configurations with two adjacent dropped crossings. The picture on the left shows the pseudo-diagram resolution in which both unresolved crossings are negative, and the picture on the right has both unresolved crossings positive. The column in the middle consists of pseudo-diagram resolutions with one positive and one negative unresolved crossing. The maps depicted are all of the possible }
\label{drop2_maps}
\end{figure}
\end{defn}

Note that for $v$ corresponding to a configuration on the left hand side of Figure \ref{drop2_maps}, the possible values of $\pi$ are $0$, $2$, and $4$. For the middle column configurations they are $0$ and $\pm 2$, and for the configuration on the right, they are $0$, $-2$, and $-4$.

By construction, $\pi(v,w) = \pi(v,u)+\pi(u,w)$.

\begin{lemma} For a cobordism from the pseudo-diagram resolution at $v$ to that at $u$, with good $\omega$, $\pi(v,u) \equiv -4p_1(P_\Delta)[X_\Delta] \pmod{8}$.

Moreover, if $N'$ has one more dropped crossing than $N$, (ie $N$ is all the crossings and $N'$ is all but one, or $N$ is all but one, and $N'$ is all but two, the one missing in $N$ and another adjacent one of opposite sign), consider the complex over $\zz \times \{0,1\}^|N'|$, with vertical cobordisms as in \cite{KM_unknot}. We can still define $\pi(v,u)=\pi(u)-\pi(v)$. In this situation, we still have $-4p_1(P_\Delta)[X_\Delta] \pmod{8}$. 
\label{pi_p1_mod_8}
\end{lemma} 
\begin{proof} Let us show the second statement only; the first follows.

We begin by showing it for vertical maps, that is the one corresponding to the extra dropped crossing in $N'$. Since both $p_1$ and $\pi(v,u)$ are additive, it suffices to show $\pi(v,u) \equiv -4p_1(P_\Delta)[X_\Delta] \pmod{8}$ for cobordisms of length $1$ or $2$, with the ones of length $2$ being from a split followed by a merge where the middle term is killed by $\omega$.

\begin{itemize}

\item  If the cobordism is length 1 and is a merge or a split, where neither the source nor the target is killed by $\omega$: the cobordism corresponds to splitting into parts each of which has an even number of endpoints of $\omega$, which does not affect the contribution to $\pi(v)$ for any unresolved crossing, so $\pi(v,u)=0$. On the other hand, the cobordism is a pair of pants, which, upon having ends capped off, becomes a sphere, for which $-4p_1(P_\Delta)[X_\Delta]  \equiv 0 \pmod{8}$.

\item  If the cobordism is length 1, and it is between two resolutions with the same number of unresolved crossings and the same number of components: because we are looking at a vertical map, the number of unresolved crossings must be $1$, so the cobordism is between a component with one negative crossing and a component with one positive crossing (along with some cylinders for the other components). 

If the cobordism goes from negative to positive, then it is a $\rrp^2_+$ with ends. By the computation in section 2.7 of \cite{KM_unknot}, $\rrp^2_+$ has two possibilities for the singular bundle data. In the non-trivial case, $p_1(P_\Delta)[X_\Delta] \equiv \frac{1}{2} \pmod 2$, so $4p_1(P_\Delta)[X_\Delta]=2$. 

Note that the fact that $\omega$ is good means that it is not possible that $s_\omega(c)=1$ for the unresolved crossing in both the source and the target, so the only possibilities are if in both $s_\omega(c)=0$, in which case $\pi(v,u)=0$ and $4p_1(P_\Delta)[X_\Delta]\equiv 0\pmod 8$, or if $s_\omega(c)=1$ on one of the sides and $0$ on the other, in which case $\pi(v,u)=-2$, and $4p_1(P_\Delta)[X_\Delta]\equiv -2 \pmod 8$, as desired.

Similarly, if the map goes from positive to negative, then the cobordism is a twice-punctured $\rrp^2_-$, and the same argument applies with the signs reversed.

\item If the cobordism is length 1 and preserves the number of components but changes the number of unresolved crossings: Let $c$ be the crossing that is unresolved in exactly one of the source and the target. Then only $c$ contributes to $\pi(v,u)$. Moreover, if $c$ is a positive unresolved crossing in the source, or a negative one in the target, then the cobordism is $\rrp^2_-$ and otherwise it is $\rrp^2_+$. Moreover the singular bundle data is nontrivial if and only if $s_\omega(c)=1$ for the unresolved projection. The computation is now similar to the previous case. 

\item  If the cobordism is length 2: By Lemma 7.2 of \cite{KM_unknot}, the composite cobordism $S_{20}=S_{10}\circ S_{21}$ is $(I \times S^3, V_{20}) \# (S^4, \rrp^2)$, where $V_{20}$ is the reverse of $S_{32}$. The $(B^4, \rrp^2)$ in this decomposition is localised around $c$, and the singular bundle data may be taken to be trivial there. Thus, the calculation for this case is the same as that for the previous two cases, but with signs reversed.

\end{itemize}

For horizontal maps, we can use the vertical maps to translate the horizontal so that it is confined to the $0,1 \pmod{3}$ levels (choosing the right one of the $0$ or $1$ mod 3, so that $\omega$ is still trivial, and applying the fact that the lemma holds for no dropped crossings (so all $\pi=0$) to show that it holds for $1$ dropped crossing, and then use that it holds for one dropped crossing to show that it holds for $2$.
\end{proof}

The main result of this subsection will be the use of the $q$ filtration to extract the following proposition:

\begin{propn} Let $\cC_q$ be the category of $q$-filtered finitely generated differential $\zz/2$ modules with differentials of order $ \geq 0$, whose morphisms are differential homomorphisms of order $ \geq 0$ up to chain homotopies of order $ \geq 0$. Then the isomorphisms type of the instanton complex of $(P_L,B,\omega)$ is a tangle invariant up to shift in $q$. That is, if $A$ is the $q$-filtered instanton complex corresponding to $(P_L,B,\omega)$ and $A'$ is that corresponding to $(P'_L,B,\omega)$ where $(P_L,B,\omega)$ and $(P'_L,B,\omega)$ represent the same tangle, then $A$ is isomorphic to $A'[c]$ in $\cC_q$, where $A'[c]$ denotes $A'$ with the filtration shifted by $c$.
\label{C_q_tangle_invar}
\end{propn} 

From this, we deduce that the isomorphism type of the pages of the spectral sequence corresponding to the $q$ filtration are tangle invariants, and then, comparing the $q$ filtration to the Khovanov picture, we will deduce theorem \ref{e1_tangle_invar}.

We will now show that the differential on the instanton complex has order $\geq 0$ with respect to the $q$ filtration. This is the analogue to Proposition 4.6 in \cite{KM_filtrations}.

\begin{lemma} Consider a cube of pseudo-diagram resolutions for a link projection that comes from one of the following: a full resolution for a projection, dropping all but one crossing, or dropping all but two adjacent, opposite sign crossings. Then the differentials on the corresponding instanton complex have order $\geq 0$ with respect to the $q$ filtration.
\label{q_on_cube_diffs}
\end{lemma}

\begin{proof}Note that Lemma 4.4 of \cite{KM_filtrations}, which states that the parity of the $q$ filtration on the instanton complex is constant, still applies; $\pi$ is even and our $q$ filtration differs from theirs by $\pi(v,o)$.

For an ASD connection with value $\beta_0$ at $u$ and $\beta_1$ at $v$, we have that if there is a map $\beta_1$ to $\beta_u$, on $S_{vu}$, then
\[q(\beta_0)-q(\beta_1) = Q(\beta_0)-Q(\beta_1)- \sum u+ \sum v -\frac{3}{2}\sigma(v,u)-\pi(v,u)\]
\[=8 \kappa +\frac{1}{2} S \cdot S-1 -\sum u+\sum v-\frac{3}{2}\sigma(v,u)-\pi(v,u)\]
\[=8 \kappa -1 -\sum u+\sum v-\sigma(v,u)-\pi(v,u).\]
The second equality above is from equation (6) in \cite{KM_filtrations}, which states that 
\[\dim(M_\kappa(S; \beta_1, \beta_0)) = 8 \kappa + \chi(S)+\frac{1}{2} S \cdot S + Q(\beta_1)-Q(\beta_0),\]
where $M(S; \beta_1, \beta_0)$ is the moduli space of instantons on $S$ from $\beta_1$ to $\beta_0$ and $M_\kappa(S; \beta_1, \beta_0)$ is the part with action $\kappa$.

From the fact that the parity of $q$ is constant, we see that it suffices to show $q(\beta_0)-q(\beta_1) \geq -1$. In other words, it suffices to show that for $v \geq u$
\[8 \kappa -\sum u+\sum v-\sigma(v,u)-\pi(v,u) \geq 0.\]

If there are no dropped crossings, then $\pi$ and $\sigma$ vanish, and the above statement follows from the non-negativity of the action and the fact that the differential in the instanton complex is upper triangular, that is, the maps vanish unless $v \geq u$.

Note that these are the possible values of $\pi(v)$ for the one dropped crossing case:
\begin{itemize}
\item The dropped crossing is negative: $\pi(v)=0,2$
\item The dropped crossing is positive: $\pi(v) = -2, 0$.  
\end{itemize} 
Thus, for one dropped crossing, the possible values of $\pi(v,u)$ for a cobordism $S$ from $v$ to $u$ are:
\begin{itemize}
\item If $S \cdot S = -2$, then $\pi(v,u)=0,2,4$
\item If $S \cdot S = 0$, then $\pi(v,u)=-2,0,2$
\item If $S \cdot S = 2$, then $\pi(v,u)=-4,-2,0$,
\end{itemize}
because $S \cdot S = -2$ if the cobordism goes from a diagram with a positive dropped crossing to one with a negative dropped crossing, $S \cdot S = 0$ for positive to positive or negative to negative, and $S \cdot S = 2$ for negative to positive. 

For the two dropped crossing picture, the possible values of $\pi(v)$ are:
\begin{itemize}
\item Left hand side of Figure \ref{drop2_maps}, ie, both crossings negative: $\pi(v)=0,2,4$
\item Middle column of Figure \ref{drop2_maps}, ie, one negative, one positive crossing: $\pi(v)=-2,0,2$
\item Right hand side of Figure \ref{drop2_maps}, ie, both crossings positive: $\pi(v)=-4,-2,0$
\end{itemize}

Note that these are the possible values of $\pi(v,u)$:
\begin{itemize}
\item If $S \cdot S = -4$, then $\pi(v,u)=0,2,\ldots 8$.
\item If $S \cdot S = -2$, then $\pi(v,u)=-2,0,\ldots 6$.
\item If $S \cdot S = 0$, then $\pi(v,u)=-4,-2,\ldots 4$.
\item If $S \cdot S = +2$, then $\pi(v,u)=-6,-4, \ldots 2$.
\item If $S \cdot S = +4$, then $\pi(v,u)=-8,-6,\ldots 0$.
\end{itemize}

For the differentials on the cube, $v \geq u$, and $8 \kappa \geq 0$, so if $\sigma + \pi \leq 0$, we are done. Thus, we may assume that $\sigma +\pi > 0$. Recall that by Lemma \ref{pi_p1_mod_8}, $\pi(v,u) \equiv -4p_1(P_\Delta)[X_\Delta]$ mod 8. By Proposition 2.7 of \cite{KM_unknot}, $\kappa = -\frac{1}{4} p_1(P_\Delta)[X_\Delta] + \frac{1}{16} S \cdot S$. Hence, $8 \kappa \equiv \frac{\sigma+\pi}{2} \pmod 4$, and it is easy to see in the above cases that $\sigma+\pi \leq 4$, so $\frac{\sigma+\pi}{2} \leq 2$. Then, if $8 \kappa < \frac{\sigma+\pi}{2}$, by the mod 4 computation, we would have to have $8 \kappa \leq \frac{\sigma+\pi}{2}-4<0$, a contradiction. Thus, $8 \kappa \geq \frac{\sigma+\pi}{2}$. 

It therefore suffices to show that $\frac{\sigma+\pi}{2}  \leq \sum v- \sum u$. Note that if $v =u$, then $\sigma=\pi=0$, and the inequality is true. So we may assume that $\sum v- \sum u \geq 1$. But $\frac{\sigma+\pi}{2}$ takes values $0,1,2$, so we only need to show that for $\frac{\sigma+\pi}{2}=2$, we have $\sum v- \sum u \geq  2$.

We do this by going through the cases. If $|S \cdot S|=4$, then obviously you need to take at least two steps.

If $S \cdot S =  2$, then to have $\sigma+\pi = 4$, we must have $\pi=2$, but if $\sum v - \sum u = 1$, then the map is an $\rrp^2$ and $S \cdot S = 2$ means that it is specifically an $\rrp^2_+$, so $\pi \equiv - 4 p_1(P_\Delta)[X_\Delta] \equiv -2 \pmod 8$, a contradiction.

Similarly, if $S \cdot S =  -2$, then to have $\sigma+\pi = 4$, we must have $\pi=6$, but if $\sum v - \sum u = 1$, then the map is an $\rrp^2$ and $S \cdot S = -2$ means that it is an $\rrp^2_-$, so $\pi \equiv - 4 p_1(P_\Delta)[X_\Delta] \equiv 2 \pmod 8$, a contradiction.

Finally if $S \cdot S =  0$, then to have $\sigma+\pi = 4$, we must have $\pi=4$, but if $\sum v - \sum u = 1$, then the map is a pair of pants, so $\pi \equiv 0 \pmod 8$.

\end{proof}

Our approach to proving Proposition \ref{C_q_tangle_invar} will be to show invariance of Reidermeister moves performed away from a ball containing $\omega$, by way of showing that dropping two adjacent, opposite sign crossings does not affect the isomorphism type in $\cC_q$ and then performing isotopies between different projections with crossings dropped. 

Observe that isotopies preserve the isomorphism type in $\cC_q$, as in the following lemma, which is analogous to Proposition 5.1 in \cite{KM_filtrations} and has the same proof, namely by considering maps $T_{vu}$ for $v \geq u$ coming from counting instantons on the trace of the isotopy from $L$ to $L'$ and showing that the chain maps and homotopies preserve the $q$ grading, as in the previous lemma.

\begin{lemma} Let $(L,N, \omega)$ and $(L',N, \omega)$ be pseudo-diagram resolutions with either no crossings dropped, one crossing dropped, or two adjacent, opposite-sign crossings dropped, and suppose that $L$ and $L'$ are isotopic via an isotopy that is constant around $N$ and $\omega$. Then $C(P_L,N,\omega)$ and $C(P_{L'},N,\omega)$ are isomorphic as elements of $\cC_q$, up to overall shift in $q$.
\label{q_isotopy}
\end{lemma}

We can also extend the complex beyond the cube to $\zz^n$, and it will be useful to note that when we extend the complex beyond the cube in one direction, there is a certain $3$ periodicity. That is, if $K_v$ and $K_u$ are the links corresponding to vertices $v$ and $u$ with $3|v-u \in \zz^n$, then $K_v =K_u$ and the $C_v$ and $C_u$ can be identified via isomorphisms with $H_*(S^2)^{\otimes p}$ where $p$ is the number of components, as in equation (2) of \cite{KM_filtrations}.

The analogue of Lemma 4.3 or \cite{KM_filtrations}, which states that the aforementioned isomorphism between $C_v$ and $C_u$ preserves the $q$ filtration, still holds in our setting, because $\pi(v)$ and $\pi(u)$ are the same, so if $v' = v+(3,0,0\ldots 0)$, then $\sigma(v',u) = \sigma(v,u)+2$ and $\pi(v,u)=\pi(v',u)$.

\begin{defn} Consider the extended complex over $ \zz \times \{0,1\}^n$. Let the type of a cobordism denote $v(0)-u(0)$. 
\end{defn}

\begin{lemma} Consider $(N,N')$ where either $N$ is all the crossings and $N'$ has one dropped crossing, or $N$ has one dropped crossing, and $N'$ has another, adjacent, opposite crossing dropped. Let $n = |N'|$ and index the crossings of $N$ $0,1,\ldots,n$, so that the $0$th one is the distinguished crossing dropped in $N'$. 

The differentials on the instanton complex over $\zz \times \{0,1\}^n$ (where the $-1 \pmod 3$ pages correspond to leaving the $0$th crossing unresolved) of type at most 3 have order $\geq 0$ with respect to the $q$ filtration.
\label{q_diffs_extended}
\end{lemma} 
\begin{proof} The proof of this is similar to the proof of Lemma \ref{q_on_cube_diffs}: First note that the parity of $q$ on the extended complex is still constant, by the same argument as before. Thus, it again suffices to show $8 \kappa -\sum u+\sum v-\sigma(v,u)-\pi(v,u) \geq 0$. This is again clear for $\sigma+\pi \leq 0$, so we may assume $\sigma+\pi >0$.

Let the vertical part of a map on the cube be the $\zz$ part and the horizontal part by the part on $\{0,1\}^n$. Then for the horizontal part, there is at most one crossing, and we have, per the above chart that $\sigma_h+\pi_h  \in \{-2,0,2\}$. For the vertical part, for a map of Type $\leq 3$, we have $|\sigma_v| \leq 2$ and $|\pi_v| \leq 2$, so we still have $\sigma+\pi \leq 6$. Thus, for $8 \kappa  \geq 0$ and $8 \kappa \equiv \frac{\sigma+\pi}{2} \pmod{4}$, we still have $8 \kappa \geq \frac{\sigma+\pi}{2}$.

It again suffices to show that $\sigma+\pi \leq 2 (\sum v- \sum u)$. For maps of Type 1, the chart in the proof of Lemma \ref{q_on_cube_diffs} still holds: to see this we need to understand the vertical cobordisms between the resolution where one of the crossings is dropped and the resolutions where it isn't.

In the case of pairs of pants, $\sigma$ and $\pi$ are both $0$. Otherwise, the cobordism is between a loop and that loop with one extra crossing, say of sign $c$. As in all of the cases on the outer rim of Figure \ref{drop2_maps}, if there is an extra crossing in the picture, we can ignore it when analysing the cobordism up to isotopy. The cobordism is between a loop and that loop with one extra crossing, say of sign $c$ is isotopic to that between a loop of sign $-c$ and a loop with sign $c$. Thus, for this case we also have $\pi = 0$ or $\pi = -\sigma$, so the values $\pi$ takes, indexed by $\sigma$, are still as in the chart in the proof of Lemma \ref{q_on_cube_diffs}.

Moreover, then the same proof applies to show $\sigma+\pi \leq 2 (\sum v- \sum u)$.

For a map of Type 2, let us find the possible values of $\sigma +\pi$ for the vertical part ie, the cobordism from $v'=v+(2,0,\ldots 0)$ to $v$. Consider $v'' = v-(1,0,\ldots 0)$. Then $\sigma(v',v) + \sigma(v,v'') = 2$ and $\pi(v',v)+\pi(v,v'')=0$. Moreover $(v,v'')$ is a Type 1 vertical cobordism, so the possibilities for $\sigma(v,v'')$ are $2,0,-2$, and the possibilities for $\pi(v,v'')$ are $-\sigma(v,v''),0$. 

The way the resolution works, for the map $(v,v'')$ which, we recall, goes between a loop and a loop with an extra crossing, the map has to go from either no crossing to positive crossing, or from negative crossing to no crossing. (Unless it is a pair of pants.) Thus, $\sigma(v,v'')$ can be $0$ or $2$, with $\pi(v,v'')=0$ if $\sigma(v,v'')=0$, and $\pi(v,v'') \in \{0,-2\}$ if $\sigma(v,v'')=2$.

Thus $\sigma(v',v)$ can be $2$ or $0$, with $\pi(v',v)=0$ if $\sigma(v',v)=2$ and $\pi(v',v)\in \{2,0\}$ if $\sigma(v',v)=0$. Either way, $\sigma(v',v)$ takes values $0,2$. From here, the same argument as in Lemma \ref{q_on_cube_diffs} works.

Finally, in the case of Type 3, we have $\sigma = 2$ and $\pi = 0$, and the same proof holds.

\end{proof}

Let $P_L$ be a link projection and $N$ be either the set of all crossings or the set of all but one crossing. Let $N' \subset N$ be obtained by dropping one crossing; in the case that $N$ already has a dropped crossing, we further require that the second dropped crossing be adjacent to the first with opposite sign. We will call the pair $(N,N')$ of sets of crossings ``okay'' if it is one of the aforementioned two situations.

Recall from \cite{KM_unknot} that in this situation, the instanton complexes $C(P_L,N)$ and $C(P_L,N')$ are quasi-isomorphic. Let us describe this quasi-isomorphism.

Let $c_* \in N$ denote that crossing that is dropped in $N'$. Note that we can decompose $C(P_L,N)$ into the two parts based on the resolution of $c_*$, as $C(P_L,N) = C_1 \oplus C_0$. We may then consider the complex $\oplus_{i \in \zz} C_i$. Then $C(P_L,N')$ is isomorphic to $C_{-1}$, and $C_{-1}$ is homotopic to $C(P_L,N) = C_1 \oplus C_0$ via maps
\[\Psi = [F_{1,-1},F_{0,-1}]:C_1 \oplus C_0 \to C_{-1}\]
and
\[\Phi_2 = [F_{2,1},F_{2,0}]: C_2 \to C_1 \oplus C_0\]
\[\Phi_{-1} = [F_{-1,-2},F_{-1,-3}]: C_{-1} \to C_{-2} \oplus C_{-3}\]
where $C_2 \simeq C_{-1}$, and the $F_{ij}$ are the maps on the instanton complex $C_i$ to $C_j$.

For the composite $\Psi \circ \Phi_{2}: C_2 \to C_{-1}$, it is shown in \cite{KM_unknot} that
\[F_{2,-1} F_{2,2} + F_{-1,-1} F_{2,-1}  + \Psi \circ \Phi = T_{2,-1} + N_{2,-1}\]
where $T_{2,-1}$ is an isomorphism coming from cylindrical cobordisms and $N_{2,-1}$ is chain homotopic to zero via a map $H_{2,-1}$ which we will describe in more detail in the proof of the following lemma.

The other composite $\Phi_{-1} \circ \Psi$ is shown in \cite{KM_filtrations} to be homotopic via chain homotopy 
\[L = \mattwo{F_{1,-2}}{F_{0,-2}}{F_{1,-3}}{F_{0,-3}}\]
to a map 
\[\mattwo{T_{1,-2}+N_{1,-2}}{0}{Y}{T_{0,-3}+ N_{0,-3}} : C_1 \oplus C_0 \to C_{-2} \oplus C_{-3},\]
which is in turn homotopic via chain homotopy 
\[\mattwo{H_{1,-2}}{0}{0}{H_{0,-3}},\]
to a map
\[\mattwo{T_{1,-2}}{0}{X}{T_{0,-3}} : C_1 \oplus C_0 \to C_{-2} \oplus C_{-3},\] for a map $X:C_0 \to C_{-2}$.

\begin{lemma} If $P_L$ is a link projection and $(N,N')$ is okay, then the instanton complexes for $(P_L,N)$ and $(P_L,N')$ are isomorphic in $\cC_q$, up to an overall shift in the $q$ filtration.
\label{q_drop_crossing}
\end{lemma}

\begin{proof} We would like to show that the morphisms $\Psi$ and $\Phi$ as well as all the homotopies in the above discussion have order $\geq 0$ with respect to the $q$ filtration. Note that $\Phi$, $\Psi$, and the chain homotopy $L$ all come from differentials of type at most three on the chain complex, which are shown in Lemma \ref{q_diffs_extended} to have order $\geq 0$.

In the case of the map $H_{2,-1}$, we write down the map, from pages 106-107 of \cite{KM_unknot}. For $C_{2} \to C_{-1}$ This map works like this: Consider $(\ol{W},\ol{S})$, which is obtained from $(W,S)$ by removing $(B_{2,-1},M_{2,-1})$ where $B_{03}$ contains basically the three handles of the $0$th crossing, and $M_{2,-1}$ is the plumbing of two M\"obius bands. The boundary of $M_{2,-1}$ is a two component unlink. Attach back in $(B_{03}, D^2 \cup D^2)$, the two disks.

For going from $(2,v)$ to $(-1,u)$, the cobordism is now $S_{(2,v)(2,u)}$, ie what it would be if we removed the part corresponding to the additionally dropped crossing.

Consider the family of metrics where you move the crossings $1,\ldots n$ back and forth, and also stretch along the boundary of $B_{03}$, and also don't quotient anything. The dimension of this family is $|v-u|_1+1$, and $\chi(\ol{S})=-|v-u|_1$, so 
\[Q(\beta_0)-Q(\beta_1)=8 \kappa+\chi(\ol{S})+\frac{1}{2}{\ol{S} \cdot \ol{S}}+\dim(G) = 
8 \kappa+\frac{1}{2}{\ol{S} \cdot \ol{S}}+1\]
and
\[q(\beta_0)-q(\beta_1)=Q(\beta_0)-Q(\beta_1)- \sum u+ \sum v +3 -\frac{3}{2}S_{(2,v)(-1,u)}\cdot S_{(2,v)(-1,u)}-\pi(v,u)\]
where $\pi(v,u)$ is three periodic.
This
\[=8 \kappa+\frac{1}{2}{\ol{S} \cdot \ol{S}}+1- \sum u+ \sum v +3 -\frac{3}{2}S_{(2,v)(-1,u)}\cdot S_{(2,v)(-1,u)}-\pi(v,u)\]
\[=8 \kappa+\frac{1}{2}{\ol{S} \cdot \ol{S}}+1- \sum u+ \sum v  -\frac{3}{2}\ol{S} \cdot \ol{S}-\pi(v,u)\]
\[=8 \kappa+1- \sum u+ \sum v  -\ol{S} \cdot \ol{S}-\pi(v,u)\]
But this is better than what we had before. Thus $H_{2,-1}$ also has order $\geq 0$.

Thus we have that $\Phi$ and $\Psi$ are morphisms in $\cC_q$, and in this category,
\[\Phi \circ \Psi_2 \simeq T_{2,-1} \text{, and }\]
\[\Phi_{-1} \circ \Psi \simeq \mattwo{T_{1,-2}}{0}{X}{T_{0,-3}},\]
and by lemma \ref{q_isotopy}, the maps $T_{1,-2}$ and $T_{0,-3}$, which correspond to isotopies, are isomorphisms in $C_q$, so $\mattwo{T_{1,-2}}{0}{X}{T_{0,-3}}$ is also an isomorphism in $\cC_q$. 

This shows that $C(P_L,N) \simeq C_{-1}$ in $\cC_q$. However, $C_{-1}$ and $C(P_L,N')$ represent the same complex, up to a constant shift in $q$. Thus, $C(P_L,N)$ and $C(P_L,N')$ are isomorphic in $\cC_q$, as desired.
\end{proof}

At this point, we can prove Proposition \ref{C_q_tangle_invar}.

\begin{proof}[Proof of Proposition \ref{C_q_tangle_invar}] We would like to show that the $\cC_q$ type is a tangle invariant. For this, it suffices to show that Reidermeister moves performed away from $\omega$ preserve the isomorphism type of the complex in $\cC_q$. This follows the proof of Proposition 8.1 of \cite{KM_filtrations}:

We compare the complexes $C$ and $C'$, obtained from cubes of resolutions corresponding to projections $P$ and $P'$ of a link that differ by a Reidermeister move performed away from $\omega$. Consider the complexes $C''$ and $C'''$ arising from the cube of pseudodiagram resolutions obtained by dropping the one or two relevant crossing in $C$ and $C'$ respectively.

Aa consequence of Lemma \ref{q_drop_crossing}, $C''$ has the same $\cC_q$ type as $C$ and $C'''$ has the same $\cC_q$ type as $C'$, and by Lemma \ref{q_isotopy}, $C''$ and $C'''$ have the same $\cC_q$ type, completing the proof.
\end{proof}

Using Proposition \ref{C_q_tangle_invar}, we now deduce Theorem \ref{e1_tangle_invar}.


\begin{proof}[Proof of Theorem \ref{e1_tangle_invar}] Theorem 3.5 in Chapter XI of \cite{Maclane} states that homotopy equivalences of order $\geq t$ induce isomorphisms of the $E^r$ pages of the spectral sequences for $r >t$. Moreover by Proposition \ref{C_q_tangle_invar}, if $P_L$ and $P_L'$ represent the same tangle, then there they are isomorphic in $\cC_q$, which means there is a homotopy equivalence between the (up to overall shift in $q$). Combining these two results, we see that the $(E_1,d_1)$ page of the instanton complex filtered by the $q$-filtration, up to overall shift in the $q$ filtration, is a tangle invariant.

By the definition of the spectral sequence, as in \cite{Maclane}, it is easy to see that the isomorphism type of the $E^1$ page is the same as the homology of the instanton complex with the differential replaced with the $\Delta_q = 0$ part of the differential, i.e., the part of the differential that changes the $q$ grading by $0$. Indeed, $E^1_p = H(F_pA/F_{p+1}A)$, where we have adjusted Theorem 3.5 in Chapter XI of \cite{Maclane}, because we are considering descending rather than ascending filtrations.

Unpacking the definition for $E^1_p$ of a filtered complex, as in the definition given in Theorem 3.1 of Chapter XI of \cite{Maclane}, $E^1_p$ is the homology of the part of $q$-grading $p$ with the $q$-grading $0$ part of the differential.

In the fully resolved picture, $\pi$ and $\sigma$ both vanish, and if $\beta_0$ is at vertex $u$ and $\beta_1$ at vertex $v$ and the coefficient of $\beta_0$ in $d\beta_1$ does not vanish, then
\[q(\beta_0)-q(\beta_1) = 8 \kappa -1-\sum u+\sum v\]
Thus, this piece of the differential has $q$-order $0$ if and only if
\[\sum v- \sum u = 1 - 8\kappa.\]
However, $8\kappa$ is non-negative, so this implies that the map is part of an edge map.

It now suffices to show that the edge maps all have $\Delta_q = 0$. However, the edge maps are calculated in Lemma 8.7 of \cite{KM_unknot}, and it is easy to see that these have $\Delta_Q = -1$, so that $\kappa=0$ and $\Delta_q = 0$, as desired.

\end{proof}

\section{Spectral sequence collapse}

In the previous section, we defined a complex, $C(P_L,\partial \omega)$, for alternating link projections, and we showed that its homology was an invariant of $(L,\partial \omega)$, and indeed independent of $\omega$.

In \cite{KM_unknot}, Kronheimer and Mrowka exhibited a spectral sequence for $(L, \emptyset)$ whose $(E_1, d_1)$ page is the Khovanov complex which abuts to $I^\#(L,\emptyset)$. 

They did this by exhibiting a spectral sequence for a link $L$ with $\omega$ whose $E_1$ term is 
\[\oplus_{v' \in \{0,1\}^N} I^\omega_*(Y,L_{v'})\]
which abuts to $I_*^\omega(Y, L_{w})$ where $w = (2,2, \ldots 2)$, so that $L_{v'}$ goes through the $0$ and $1$ resolutions of a link $L$, and $L_{w}$ is the unresolved link. They then showed that for unlinks $L_{v'}$ with $n$ components, in the situation where $\omega$ is empty, $I^{\#}(Y,L_{v'})$ is the group $A^{\otimes n}$, and that the maps $d_1$ agree with those in the Khovanov complex.

It is easy to see that for general $\omega$ and $L_{v'}$ an unlink, $I^{\#}(Y,L_{v'}, \omega)$ agrees with $C(P_L,\partial \omega)$ with $d_1$ also agreeing with the differential of $C(L,\omega)$. 

This leads us to the following theorem:

\begin{thm} For an alternating link projection $P_L$ with singular bundle data $\omega$, there is a spectral sequence whose $(E_1,d_1)$ term is $C(P_L,\partial \omega)$, which abuts to $I^\#(Y,L,\omega)$. 
\end{thm}

In \cite{KM_unknot}, Kronheimer and Mrowka also showed that for $K$ an alternating knot, the spectral sequence from Khovanov homology to instanton homology collapses on the $E_2$ page. This means that the Khovanov homology and the instanton homology have the same rank for a alternating knot projection $P_K$. By corollary \ref{indep_slide_omega}, for an alternating knot projection $P_K$ with $\omega$, the homology of $C(P_K, \omega)$ is the same as the Khovanov homology of $K$. This implies the following:

\begin{lemma}[Corollary 1.6 from \cite{KM_unknot}] For an alternating knot projection $K$ with marking data, the spectral sequence from $C(P_K, \omega)$ to the instanton homology collapses on the $E_2$ page.
\label{arc_knot_spec_collapse}
\end{lemma}

\begin{proof} To avoid confusion, let us spell out the reasoning of Corollary 1.6 from \cite{KM_unknot}. In \cite{KM_unknot}, Kronheimer and Mrowka show that, with $\zz$ coefficients, for any link $L$, there is a spectral sequence with $E_2$ term the reduced Khovanov homology, $Khr(\ol{L})$, which abuts to the reduced instanton homology $I^{\natural}(L)$. They further showed that with $\qq$ coefficients, the reduced singular instanton homology $I^\natural(L,\qq)$ is isomorphic to the sutured Floer homology $KHI(L;\qq)$. 

They also showed in \cite{KM_KHI_Alexander} that for a {\em knot}, the rank of the sutured Floer homology $KHI(K;\qq)$ is the sum of the absolute values of the coefficients of the Alexander polynomial of $K$. Thus, for quasi-alternating knots $K$, the rank of $KHI(K;\qq)$, and therefore that of $I^\natural(K,\qq)$ is bounded below by the determinant of $K$.

In \cite{MO_quasi}, Manolescu and Ozsv\'{a}th showed that the rank of the Khovanov homology for a quasi-alternating link is equal to the determinant. Thus, for quasi-alternating knots, the rank of $I^\natural(K,\qq)$ is equal to that of $Khr(\ol{K},\qq)$. 

Moreover, it was shown in \cite{MO_quasi} that the reduced Khovanov homology over $\zz$ is a free $\zz$ module. Thus, in the instanton complex, the $E_2$ page is a free $\zz$ module, and for the $E_\infty$ page to have the same rank over $\qq$ as the $E_2$ page, which we just proved must hold, the differentials on the $E_2$ page and beyond must vanish over $\zz$. Thus, the spectral sequence collapses on the $E_2$ page over $\zz$, and therefore over $\zz/2$, as desired.
\end{proof}

Let us now extend this result to alternating links:

\begin{thm} For a non-split alternating link projection $P_L$ with marking data $\omega$, the spectral sequence from $C(P_L,\partial \omega)$ to the instanton homology collapses on the $E_2$ page.
\label{spec_collapse}
\end{thm}

\begin{cor} For a non-split alternating link $L$, the rank of the instanton homology $I^\#(L,\omega)$ is independent of $\omega$.
\end{cor}

\begin{proof}[Proof of Corollary] By Theorem \ref{spec_collapse}, the instanton homology of $I^\#(L,\omega)$ has the same rank as the homology of $C(P_L,\partial \omega)$. The latter, however, has the same rank as the homology of $C(P_L,\emptyset)$ by Corollary \ref{indep_omega}, which by Theorem \ref{spec_collapse}, also agrees with $I^\#(L,\emptyset)$. Thus the homology of $I^\#(L,\omega)$ has the same rank as that of $I^\#(L,\emptyset)$
\end{proof}

\begin{proof}[Proof of Theorem \ref{spec_collapse}] In the course of this proof, we are returning to the notation in subsection \ref{alternating_marked_points_section}, where the maps go from the $0$ resolution to the $1$ resolution.

Recall that for an $n$ dimensional cube of resolutions, we have associated $(C, f)$. Let us start by describing this complex; in doing so we will set up the notation for this section. Consider the cube of resolutions associated to the link projection. For $v$ a vertex of the cube,  let $C_v$ be $A^{\otimes n}$ or $0$, where $n$ is the number of components, realised as $I^\omega(S^3, L_v^\natural)$. For $u,v \in \{0,1\}^n$ with $v \geq u$, let the map $f_{uv}:C_u \to C_v$ count instantons on the cobordism from $L_u$ to $L_v$. The cube $C$ is defined to be $C = \oplus_v C_v$, and the maps on it are the $f_{uv}$. 

Here $f_{vv} = 0$, for reasons of degree, and for $|v-u|=1$, $f_{uv}$ is merge or split map as in Khovanov homology when $\omega$ is trivial at the source and target.

We group the complex and differentials by Khovanov cohomological degree; that is, for $i \geq 0$, let $C_i$ denote $\oplus_{|v|=i} C_v$ and for $x \in C$, let $\partial_i(x)$ denote the part of $\partial(x)$ in $C_i$. Let $F_p$ denote $\oplus_{i \geq p} C_i$, so that $F$ is a descending filtration on $C$.

We start with the following Lemma, which reformulates what it means for the spectral sequence to collapse:

\begin{lemma} To say that ``the spectral sequence collapses on the $E_2$ page'' for a link projection singular bundle data $(P_L,\omega)$ is the same as saying that for any $r \geq 2$ and $x \in F_p$ such that $\partial_{p+1} x, \partial_{p+2} x, \ldots,  \partial_{p+r-1} x = 0$, then there is $y \in F_{p+1}$ such that $ \partial_{p+2} y, \partial_{p+3} y,\ldots, \partial_{p+r-1} y = 0$ and $\partial_{p+r}y= \partial_{p+r} x$. 
\end{lemma}
\begin{proof}  By the definition of the spectral sequence, as in the proof of theorem 3.1 of chapter XI of \cite{Maclane},
\[E^r_p = \eta_p(\{x \in F_p| \partial_{p+i}(x) = 0 | i<r\})/\eta_p(\partial(\{x \in F_{p-r+1}| \partial_{p-r+1+i}(x) = 0 | i<r-1\}))\]
where $\eta_p$ is the projection $F_p \to F_p/F_{p+1}$. The spectral sequence differential $d^r: E^r_p \to E^r_{p+r}$ is the map induced by $\partial$. The lemma follows from unpacking the definition of $d^r$. 
\end{proof}

We now show show that given that the spectral sequence collapses on the $E_2$ page for alternating knots, it also collapses similarly for alternating links, regardless of $\omega$, by induction on the number of components. The base case is the statement that the spectral sequence collapses for alternating knots, Lemma \ref{arc_knot_spec_collapse}; in this case, since there is only one component, $\omega$ is always trivial, and therefore does not affect $I^\sharp(L,\omega)$. We have shown that it does not affect $\Kh(P_L,\omega)$ earlier.

Assume that the claim holds for alternating links of $l$ components. Consider link $L$ with $l+1$ components.

Let $n$ be the number of crossings, indexed $1,\ldots, n$. Consider some crossing $k$ where the two strands are from different components; this exists because $L$ is not split. Without loss of generality, let $k=1$.

Consider the alternating link $L'$ formed by taking $L$ and adding another crossing right next to $k$, between the same two strands, as in Figure \ref{L_to_L_prime}. Note that there are two ways to do this, depending on which side you add the new crossing. In one of these, it will be the case that the $0$ resolution of the new crossing in $L'$ is the same as $L$, in the other it will be the $1$ resolution of the crossing that gives $L$; choose the former of the two.

\begin{figure}[ht!]
\centering
\includegraphics[width=60mm]{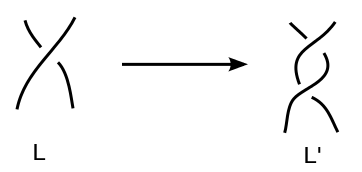}
\caption{A depiction of adding a crossing. For example, the usual projection of the Hopf link would be turned into a trefoil.}
\label{L_to_L_prime}
\end{figure}

Let the new crossing be indexed $0$; that is, the crossings of $L'$ are labelled by $0,1,2,\ldots$, where $0$ is the new crossing, and the others are the same as the corresponding crossing in $L$.

Let $C$ and $C'$ be the complexes for $L$ and $L'$, respectively. Let $V_x$ and $V_x'$ be the sets of vertices of degree $x$ for $L$ and $L'$, respectively. (See Figure \ref{figure10}.)

\begin{figure}[ht!]
\centering
\includegraphics[width=120mm]{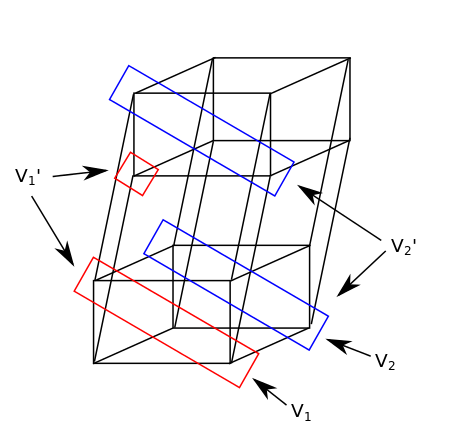}
\caption{}
\label{figure10}
\end{figure}

The cube for $L'$ consists of a bottom cube, that is a cube with $(0, *)$, ie coordinate $0$ in index $0$, and a top cube, $(1, *)$. The bottom cube is isomorphic to $L$ with the same edge and diagonal maps, as in the following claim. We can consider $C_i \subset C'_i$ as the bottom cube, but this inclusion is not a map of complexes.

Observe that relative to the other crossings, the $0$ and $1$ crossing on $L'$ look the same, so that on $L'$,
\[C'_{(0,1,x_2,x_3,\ldots)}  \simeq C'_{(1,0,x_2,x_3,\ldots)}\]
and 
\[d'_{{(0,0,x_2,x_3,\ldots)} \to {(0,1,x_2,x_3,\ldots)}} = d'_{{(0,0,x_2,x_3,\ldots)} \to {(1,0,x_2,x_3,\ldots)}},\]
respecting the above isomorphism, for $x_i \in \{0,1\}$.

\begin{claim} In the cube for $L'$, for integers $a,b$, the maps $f'_{ab}:C_a' \to C_b'$, restricted to the cube for $L$ seen as $(0,*)$ in the cube for $L'$, ie restricted to $C_a \to C_b$, it is the same as $f_{ab}:C_a \to C_b$.
\end{claim}
\begin{proof} For $u,v$ vertices in the cube for $L$, we wish to show that $f_{uv} = f'_{uv}$, where the right hand side is obtained from viewing $u,v$ as vertices in the cube for $L$.

But recall that the maps $f_{uv}$ come from a moduli space over a family of metrics on a cobordism from the unlinks representing $C_u$ to the unlinks repsenting $C_v$, which we call $L_u$ and $L_v$. Let $L_u'$ and $L_v'$ denote the unlinks for $(0,u)$ and $(0,v)$ as vertices in the cube for $L'$.

Then the cobordism $L_u \to L_v$ is isomorphic to the cobordism $L_u' \to L_v'$, and the families of metrics and moduli spaces are also isomorphic. The induced maps therefore agree.
\end{proof}

Consider the cube $C'$ as $C_{00}' \oplus C_{01}' \oplus C_{10}' \oplus C_{11}'$, with $d_{ij}'$ on each $C_{ij}'$ for $i,j \in \{0,1\}$, with additional maps $d_{i,j,i+1,j}'$ and $d_{i,j,i,j+1}$ between parts. So what we have is that $C$ can be thought of as $C_{00}' \oplus C_{01}'$ with $d_{00}'$, $d_{01}'$, and $d_{0,0,0,1}'$ for differentials, and $(C_{01}',d_{01}') \simeq (C_{10}',d_{10}')$. (See Figure \ref{figure11})

\begin{figure}[ht!]
\centering
\includegraphics[width=120mm]{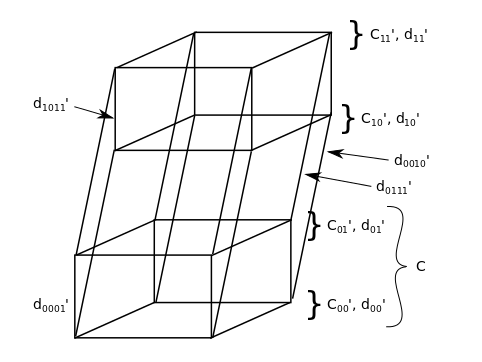}
\caption{}
\label{figure11}
\end{figure}

\begin{lemma} The map of Khovanov complexes (i.e. disregarding diagonal maps), $g:C \to C'$ given by considering $g(x,y) = (x,y,y,0)$ for $(x,y) \in C_{00}' \oplus C_{01}' \simeq C$ and $(x,y,y,0) \in C_{00}' \oplus C_{01}' \oplus C_{10}' \oplus C_{11}' \simeq C'$ is a chain map on the marked Khovanov complex, as is the map in the other direction, $h(x,y,z,w) = (x,y)$. The quotient map $h$ is, moreover, a chain map on the instanton complex.
\end{lemma}
\begin{proof} Consider $C = C_{00}\oplus C_{01}$ where the first $0$ in the index doesn't actually mean anything, but is just to keep notation convenient, and the $0$ and $1$ in the second index indicates the resolution of crossing $1$. Let $d_{00}$ and $d_{01}$ be the differentials on $C_{00}$ and $C_{01}$ respectively, and $d_{0001}$ is the  differential $C_{00} \to C_{01}$. (Note that the differentials here do not include the diagonal maps on the cubes.)

Then we know that $C_{00} = C_{00}'$, $C_{01} = C_{01}'$, $d_{00} = d_{00}'$, $d_{01} = d_{01}'$, and $d_{0001} = d_{0001}'$. Moreover, $C_{01}' = C_{10}'$, $d_{1011}' = d_{0111}'$, and $d_{0001}' = d_{0010}'$.

For $g$ to be a chain map, we want $d'g(x,y) = g(d(x,y))$, where
\[LHS = d'g(x,y) = d'(x,y,y,0) = (d_{00}'x, d_{0001}'x + d_{01}'y, d_{0010}'x + d_{10}'y, d_{0011}' x+ d_{1011}'y + d_{0111}'(y))\]
\[ = (d_{00}'x, d_{0001}'x + d_{01}'y, d_{0001}'x + d_{01}'y, 0)\]
because we are over a ring of characteristic two, and because $d_{1011}' = d_{0111}'$, $d_{0001}' = d_{0010}'$, and $d_{01}' = d_{10}'$. 

On the other hand,
\[RHS = g(d(x,y)) = g(d_{00}x, d_{0001}x + d_{01}y) = (d_{00}'x, d_{0001}'x + d_{01}'y, d_{0001}'x + d_{01}'y, 0) = LHS\]
as desired.

The proof for $h$ is similar.

The fact that $h$ is a chain map on the instanton complex follows from the fact that it is a quotient by $C'':= C'_{10} \oplus C'_{11}$, and the latter is a sub-complex for both Khovanov and instanton differentials. The diagonal maps on $C$ and $C''$ can be chosen to agree with those on $C'$ by choosing auxiliary data, such as perturbations, for the cobordisms in $C'$ and letting $C$ and $C''$ inherit these data from $C'$. 
\end{proof}

To summarise, we now have a sequence
\begin{diagram}
C'' & \rTo^i &  C' & \pile{\rTo^h \\ \lTo_g }& C & \rTo^{(d_{Kh},d_I)}&  C''
\end{diagram}
Where $i$ is the inclusion of the upper cube, $C''$, into the larger cube, $C'$, which is a chain map for both $d_{Kh}$ and $d_I$, and $h$ is the map described above, projecting to the lower cube, also a chain map for both $d_{Kh}$ and $d_I$. Thus $C'$ can be seen as the mapping cone for the map $C \to C''$ for either the Khovanov or the instanton differentials.

The map $g$ is a splitting of the mapping cone for the Khovanov differential.

\begin{lemma} Consider a filtered complex $C'$ viewed as a filtered mapping cone of $d_I: C \to C''$, such that there is a splitting $g$ of the projection $h$ on the $E_1$ page, so that $C,C',C''$ fit into
\begin{diagram}
C'' & \rTo^i &  C' & \pile{\rTo^h \\ \lTo_g }& C & \rTo^{d_I}&  C''.
\end{diagram}
Then if the spectral sequence for $(C',d_I')$ collapses on the $E_1$ page, then the same holds for $(C,d_I)$.
\end{lemma}

\begin{proof} We wish to show that for $r \geq 2$ and $x \in F_pC$ such that $\partial_{p+1}x, \partial_{p+2}x, \ldots, \partial_{p+r-1}x =0$ then there is $y \in F_{p+1}C$ such that $\partial_{p+2}y, \ldots \partial_{p+r-1}y = 0$, and $\partial_{p+r}y=\partial_{p+r}x$.

Let $x \in F_pC$ be such that $\partial_{p+i} x = 0$ for $i<r$. Because $\partial_{p+1}x=0\in F_{p+1}C/F_{p+2}C$, applying the above lemma we have $\partial_{p+1}'g(x)=0 \in F_{p+1}C'/F_{p+2}C'$; here the differentials $d$ and $d_1$ agree because when considering $\partial_{p+1}$ on $F_p$, only the $d_1$ maps come into the picture.

By the assumption of spectral sequence collapse on $L'$, there is $y_{p+1}' \in F_{p+1}C'$ such that $\partial_{p+2}'(g(x)) = \partial_{p+2}' (y_{p+1}') \in F_{p+2}C'/F_{p+3}C'$. 

Consider $g(x)+y_{p+1}'$, we have $\partial_{p+1}'(g(x)+y_{p+1}')=0$ and $\partial_{p+2}'(g(x)+y_{p+1}')=0$, so applying spectral sequence collapse again, we get that there is $y_{p+2}' \in F_{p+2}C'$ with $\partial_{p+3}'(g(x_p)+y_{p+1}'+y_{p+2}')=0$. Iterating, we get that for any $r$ there is $y' = y_{p+1}'+\cdots +y_{p+r-1}' \in F_{p+1}C'$ such that $\partial_{p+i}' (g(x)+y')=0$ for $i \leq r$. 

For $x+h(y')$ in $C$, we have $h(y') \in F_{p+1}C$ and $\partial_{p+i}(x_p+h(y'))=0$ for $i \leq r$, by applying $h$ to the statement $\partial_{p+i}' (g(x)+y')=0$ for $i \leq r$, and using that $h$ is a chain map for both $d$ and $d_1$.

Let $y = h(y') \in F_{p+1}C$. Then,
\[\partial_{p+i}(y)=\partial_{p+i}(h(y'))=\partial_{p+i}(x)\]
for $i\leq r$, as desired.

\end{proof}

Theorem \ref{spec_collapse} now follows from the above lemma.

\end{proof}

\section{Binary dihedral representations}

In this section, we study the effect of $\omega$ on binary dihedral representations $R(L, \omega)$ and $R_B(L,\omega)$, for a link $L$ with $\omega$, as defined in the introduction.

To understand these representations, let us consider a projection for the link along with $\omega$, drawn in two dimensions, and let us consider ``arcs'' in the projection, meaning continuous pieces of the drawing, where two adjacent arcs are either separated by something passing over the gap between them, or by an endpoint of $\omega$.

Let us label the arcs of the projection $a_{i,j}$ with $1 \leq j \leq j_i$, where $i$ indexes the component number, $j$ indexes the arc number on a certain component, and $j_i$ is the number of arcs on component $i$; without loss of generality let us choose a labelling for the arcs such that $a_{i,1}, a_{i,2}, \ldots a_{i,j_i}$ go along component $i$ in counter-clockwise order. Let $b_{i,j}$ denote the arcs of components of $\omega$, labelled similarly.

The index $j$ in $a_{i,j}$ will be taken mod $j_i$. (However on $b_{i,j}$ the indices are not modulo anything.)

\begin{lemma} The $\omega$ dependence of the spaces $R(L,\omega)$ and $R_B(L,\omega)$ can be reduced to dependence on the parity of the number of endpoints of $\omega$ on each component.
\end{lemma}
\begin{proof} Let $x_{i,j}$ and $y_{i,j}$ denote the meridians around $a_{i,j}$ and $b_{i,j}$ respectively. Then representations in $R(L,\omega)$ are given by the images of $x_{i,j}$, which we denote $\theta_{i,j}$ in the binary dihedral group, $N$, with constraints:
\[\theta_{i,j} \theta_{i',j'} = \theta_{i',j'}\theta_{i,j+1}\]
when $a_{i',j'}$ passes between arcs $a_{i, j}$ and $a_{i,j+1}$,
\[\theta_{i,j} \theta_{i,j+1}^{-1} = -1\]
when there is an endpoint of $\omega$ separating $a_{i, j}$ and $a_{i,j+1}$, and
\[\theta_{i,j} (-1) = (-1) \theta_{i,j+1}\]
where some arc $b_{i',j'}$ passes between arcs $a_{i, j}$ and $a_{i,j+1}$.

The last constraint would be the same if we had $\omega$ pass under instead of over the part of $L$, because it just says $\theta_{i,j} = \theta_{i,j+1}$, and if $\omega$ passed under instead of over, we would have only one $a_{i,j}$ instead of having $\omega$ split it into $a_{i,j}$ and $a_{i,j+1}$, which would have the same effect.

This shows that $R(L,\omega)$ only depends on the endpoints.

It remains to show that $R(L,\omega)$ only depends on the parity of the number of endpoints of $\omega$ on each component. It is clear that if there are two endpoints of $\omega$ on the same arc $a_{i,j}$, then we can cancel them.

It now suffices to show that dragging an endpoint of $\omega$ across a crossing of the link does not affect $R(L,\omega)$. Suppose there is an endpoint of $\omega$ separating $a_{i,j}$ and $a_{i,j+1}$. Then the relations involving $\theta_{i,j}$ and $\theta_{i,j+1}$ are 
\[\theta_{i,j-1} \theta_{i',j'} = \theta_{i',j'}\theta_{i,j} \text{ and }\theta_{i,j+1} \theta_{i'',j''} = \theta_{i'',j''}\theta_{i,j+2},\]
but $\theta_{i,j+1} = -\theta{i,j}$, because $a_{i,j}$ and $a_{i,j+1}$ are separated by an endpoint of $\omega$, so we could just eliminate $\theta_{i,j+1}$ and write the relations as:
\[\theta_{i,j-1} \theta_{i',j'} = \theta_{i',j'}\theta_{i,j} \text{ and }\theta_{i,j} \theta_{i'',j''} = -\theta_{i'',j''}\theta_{i,j+2}.\]

In this setting, we can instead look at the picture as having some marked crossings on each component, and having $\theta_{ij}$ with crossing relations $\theta_{i,j-1} \theta_{i',j'} = \theta_{i',j'}\theta_{i,j}$ at normal crossings and $\theta_{i,j-1} \theta_{i',j'} = -\theta_{i',j'}\theta_{i,j}$ at marked crossings, and we are saying that moving the markings around along $a_{ij}$ for fixed $i$ does not affect the representation, but moving a marked crossing from being between $a_{i,j}$ and $a_{i,j+1}$ to being between $a_{i,j+1}$ and $a_{i,j+2}$ is like flipping the sign of $\theta_{i,j+1}$, which is just a renaming and has no effect on $R(L,\omega)$ as desired.

\end{proof}

In the course of showing that lemma, we exhibited a different way to label arcs, which we will now adopt. Consider $a_{ij}$ the arcs of $L$, now only considered to be separated if something in $L$ passes over; that is, we are ignoring $\omega$ in this picture. For each component of $L$ that has an odd number of endpoints of $\omega$, we consider one of the crossings for which that component is the underbranch to be marked, and we have that $R(L,\omega)$ is given by $\theta_{ij} \in BD$ with relations
\[\theta_{i,j} \theta_{i',j'} = \theta_{i',j'}\theta_{i,j+1}\]
for an unmarked crossing of $a_{i',j'}$ passing over $L_i$ separating $\theta_{i,j}$ and $\theta_{i,j+1}$ and 
\[\theta_{i,j} \theta_{i',j'} = -\theta_{i',j'}\theta_{i,j+1}\]
if the crossing in question is marked.

The relations for crossings shows that if two arcs belong to the same component, their images are conjugate to each other. Note that in $BD$, elements of $S^1_B$ can only be conjugate to other elements of $S^1_B$, and the same for $S^1_A$, so each component maps entirely to one of $S^1_A$ and $S^1_B$.

\subsection{Concerning $R_B(L,\omega)$}
We restrict our attention to $R_B(L,\omega)$, i.e., the conjugacy classes of representations that take meridians of the link to $S^1_B$. Note that for $X=\cos(x) J+\sin(x) K$, and $Y = \cos(y)J+\sin(y)K$, we have
\[Y^{-1}XY = \cos(2y-x)J+\sin(2y-x)K.\]
Thus, $X_1Y=YX_2$, for $X_i = \cos(x_i) J+\sin(x_i) K$ means $x_1+x_2=2y$, modulo $2\pi$.

Moreover, for $X=\cos(x) J+\sin(x) K$, the quaternion $-X$ corresponds to angle $x+\pi$. 

Now, changing notation, we can think of arc $x_{i,j}$ as mapping to angle $\theta_{i,j} \in \rr/(2\pi\zz)$, and the relations are
\[\theta_{i,j}+\theta_{i,j+1} = 2\theta_{i',j'}\]
for an unmarked crossing, and
\[\theta_{i,j}+\theta_{i,j+1} = 2\theta_{i',j'}+\pi\]
for a marked crossing.

We can represent this as 
\[M\theta = v,\]
where $M$ is a matrix with coefficients $0$, $\pm 1$ or $\pm 2$, $\theta$ is a vector whose entries are $\theta_{i,j}$ and $v$ is a vector with coefficients $0$ or $\pi$, with the $\pi$ corresponding to marked crossings.

Let us describe $M$ more explicitly: the rows of $M$ correspond to crossings and the columns to arcs. For each crossing, its row has a $1$ for each of the two arcs that end there, a $-2$ for the arc that goes over, and $0$s elsewhere. (The reason the entries could be $-1$ or $2$ is that it is possible that some of the three arcs described could coincide, in which case the $1$s or $-2$s add.)

Representations in $R_B(L,\omega)$ now correspond to solutions to $M\theta = v$. Note that conjugating a representation shifts all entries of $\theta$ by a constant. Thus, conjugacy classes can be seen as vectors $\theta \in (\rr/2 \pi \zz)^n$ with $\theta_n=0$, with $M\theta = v$. 

Note that if $v=0$, that corresponds to $\omega$ being trivial. Moreover, if there is any solution to $M\theta=v$ with $\theta_n=0$, then the number of such solutions is the same as the number of solutions to $M\theta=0$ with $\theta_n=0$, because given one solution to the former, other solutions are obtained by translation by solutions of the latter. We immediately deduce the following:
\begin{lemma} For any link $L$, either $R_B(L,\omega) = R_B(L,\emptyset)$, or $R_B(L,\omega) = \emptyset$. 
\end{lemma}

Of course, $R_B(L,\emptyset)$ cannot be empty because the $\theta = 0$ is a solution to $M\theta= 0$ with $\theta_n = 0$, so the two cases described in the lemma are mutually exclusive.

Note that if $L = L_1 \cup L_2$ is a split link, and $L_1$ is an unknot with an endpoint of $\omega$, then $R_B(L,\omega)=0$, because the equation for it is $\theta_i = \theta_i+\pi \pmod{2 \pi}$. If $L_1$ is an unknot with no endpoints of $\omega$, then it does not add any interesting structure to $R_B(L_1 \cup L_2, \omega)$. For the rest of this section, we will only consider links that do not have a split unknot component. 

If there is a component without any crossings in which it is the underbranch, then it is clearly a split unknot component. Thus, we may assume that every component has at least one crossing. This implies that on each component, the number of crossings is equal to the number of arcs; after all, every arc has two endpoints (as long as it is not a whole component) and every crossing has two endpoints on it. Thus, the matrix $M$ is square.

We may also rearrange $M$ so that columns $1,\ldots a_1$ correspond to arcs of the first component, going (in some direction) along the component, columns $a_1+1,\ldots a_2$ correspond to the arcs of the second component, etc. We may also arrange so that rows $1,\ldots a_1$ correspond to the crossings that separate arcs $(1,2)$, $(2,3)$, ... $(a_1,1)$, and rows $a_1+1, \ldots a_2$, similarly correspond to crossings separating arcs of component $2$.

Note that this means that the diagonal of the matrix has all entries $\pm 1$ or $2$. Let $M_-$ denote the matrix obtained by deleting the $n$th row and the $n$th column of $M$. It is easy to see that the rank of $M$ is at least the rank of $M_-$. 

Moreover, the determinant of $M_-$ is the same as the determinant of the link (this could be taken to be a definition of the determinant, cf \cite{MI_virtual} page 79-80).

We can now show the following theorem:

\begin{thm} If $L$ is a link with $\det(L) \neq 0$ then $R_B(L,\omega)$ is independent of $\omega$.
\end{thm}
\begin{proof} As mentioned in the preamble to this theorem, the determinant of $L$ is the determinant of $M_-$, so since the determinant is non-zero, $M_-$ has rank $n-1$, so $M$ has rank at least $n-1$. We wish to show that $v$ is in the image of $M:(\rr/2\pi \zz)^n \to (\rr/2 \pi \zz)^n$. Recall that the entries of $v$ correspond to crossings. We have grouped them by the component of the underbranch. The entries of $v$ are $\pi$ for some crossings, and otherwise zero, where there is at most one entry that is $\pi$ for each component.

To understand the image of $M$ in $(\rr/2 \pi \zz)^n$, we analyse the image of $M$ in $\rr^n$. Note that $\im (M) = (\ker(M^T))^\perp \subset \rr^n$, where $(\ker(M^T))^\perp$ denotes the orthogonal complement of the kernel of the transpose of $M$. Since $M$ has rank at least $n-1$, to describe $(\ker(M^T))^\perp$, it suffices to find one vector in $\ker(M^T)$.

Consider the vector $u$ with entries $\pm 1$, where the sign is the sign of the crossing (where vectors in the domain of $M^T$ are thought to be indexed by crossings). We would like to show that $M^Tu=0$. Rows of $M^T$ correspond to arcs, with entries $+1$ for each of the two (counted with multiplicity) crossings in which the arc is an under-branch, and $-2$ for each crossings for which the arc is the over-branch.


Thus, on each arc, label the crossings $c_1, c_2,\ldots, c_k$, where $c_1$ and $c_k$ are the crossings at the ends, and are allowed to coincide. Here, when we count the crossings of an arc, we are including both the crossings at the end, where the arc is the underbranch, and the crossings in the middle, where the arc is the overbranch; it is easy to see that with this definition, each arc must involve at least two crossings. Then the crossings $c_2,c_3,\ldots c_{k-1}$ alternate signs. Crossings $c_1$ and $c_2$ have the same sign, and crossings $c_{k-1}$ and $c_k$ have the same sign. It is now easy to see that $u_{c_1} -2(u_{c_2}+u_{c_3}+\cdots+u_{c_{k-1}})+u_{c_k}=0$ by checking in the two cases arising from the parity of $k$. Thus, the $u$ we constructed is in $\ker(M^T)$, so the image of $M$ is the orthogonal complement of $u$.

At this point, to show that $v  \in (\rr/2 \pi \zz)^n$ is in the image of $M$, it suffices to show that there is some representative of it in $\rr^n$ which is orthogonal to $u$. Some representatives of $v$ are $(\pm \pi, 0,0,\ldots, 0, \pm \pi, 0, \ldots 0, \pm \pi, 0,\ldots)$, where there are an even number of entries that are $\pm \pi$, and we are allowed to choose the signs. It is easy to see that we can choose signs so that $v \cdot u = 0 \in \rr$, because all of the entries of $u$ are $\pm 1$, so we can choose the signs in $v$ to make the $\pi$s cancel.

\end{proof}

This theorem implies that for quasi-alternating links, $R_B(L,\omega)$ is independent of $\omega$. Moreover from the proof of it, we see that the signed sum of each column of $M$ is zero. We already know that the sum of each row is zero, because each row has two $1$s and a $-2$, (with multiplicity). Hence, deleting a column and then deleting a row do not change the rank of $M$, and we deduce that the rank of $M$ is equal to the rank of $M^-$. We will use this to show the following partial converse to the above theorem:

\begin{thm} Let $L$ be a two component link with $\det(L) = 0$. Then for $\omega$ going between the two components, $R_B(L,\omega) = \emptyset$. 
\end{thm}
\begin{proof} It suffices to show that there is an element $u \in \ker(M^T)$ such that no representative of $v$ in $\rr^n$ can be orthogonal to $u$. Note that $v = (\pi,0,0,\ldots, 0, \pi,0,0,\ldots 0) \in (\rr/2 \pi \zz)^n$ where the non-zero entries are in crossings from different components. Thus, the representatives of $v$ are of the form $(\pi+2 \pi a_1,2 \pi a_2,\ldots, 2 \pi a_{k-1}, \pi+2 \pi a_k,2 \pi a_{k+1},\ldots)$ with $a_i \in \zz$. 

It then suffices to find $u \in  \ker(M^T) \subset \zz^n \subset \rr^n$ such that the entries of $u$ are even on one component and odd on the other; with that, by a parity argument, we could see that $u$ cannot be orthogonal to any representative of $v$.

Since $\det(M)=0$, the rank of $M^-$, and therefore the rank of $M$ is at most $n-2$, that is the kernel of $M^T$ has rank at least $2$.

Consider the exact sequence
\[0 \to K \to \zz^n \oset{M^T}\to \zz^n \to C \to 0\]
where $K$ is the kernel and $C$ is the cokernel of $M^T$. We know that $K$ has rank at least $2$. However, $K$, being a submodule of $\zz^n$ is a free $\zz$ module, so this sequence is a projective resolution of $C$. Thus, the homology groups of
\[0 \to K \otimes \zz/2 \to (\zz/2\zz)^n  \oset{\ol{M^T}}\to (\zz/2\zz)^n \to 0\]
compute the $\Tor^\zz(\zz/2,C)$, where the homology at $K \otimes \zz/2$ is $\Tor^\zz_2(\zz/2, C)$. Note that $\zz$ has global dimension 1, so $\Tor^\zz_2(\zz/2, C)=0$. Thus $K \otimes \zz/2 \to (\zz/2\zz)^n$ in the above sequence is injective, and, indeed is an injection into the kernel of $\ol{M^T}$.

However, $\ol{M^T}$ consists of two diagonal blocks that look like
\[\begin{bmatrix}
    1       & 0 & 0  & 0 &  \dots & 1  \\
    1       & 1 & 0  & 0 &  \dots & 0  \\
     \vdots & \vdots &\vdots & \vdots & \ddots & \vdots \\ 
   0 & 0        & 0 &0  \dots & 1 & 0\\
   0 & 0        & 0 &0  \dots & 1 & 1\\
\end{bmatrix}\]
where one block corresponds to each component of the link, so it is easy to see that its kernel has rank $2$, with kernel generated by $(1,1,\ldots 1, 0,0,\ldots 0)$, and $(0,0,\ldots 0, 1,1,\ldots 1)$, ie, vectors that are $1$ on one component and $0$ on the other.

Thus, the injectivity of the map $K \otimes \zz/2 \to (\zz/2\zz)^n$ implies that it is an isomorphism to the kernel of $\ol{M^T}$ in $(\zz/2\zz)^n$. However, the map comes from looking at generators of $K \subset \zz^n$ modulo $2$, so we get that looking at $K$ modulo $2$ in $(\zz/2\zz)^n$, we get
\[\Span((1,1,\ldots 1, 0,0,\ldots 0),(0,0,\ldots 0, 1,1,\ldots 1)) \subset (\zz/2\zz)^n.\]
Consequently, some element of $K$ must be odd on one component and even on the other, as desired.
\end{proof}

When $L$ has more than two components, and $\det(L)=0$, there could be situations where $R_B(L,\omega)$ is the same as $R_B(L,\emptyset)$ for $\omega$ going between some components, but $R_B(L,\omega)$ is empty for other $\omega$. For example, consider the link in Figure \ref{det0_eg}. If we consider only the black, blue, and green components (ignoring the red component), we get a three component link. For this link, $\omega$ between the blue and black components and $\omega$ between the blue and green components have $R_B(L,\omega) = \emptyset$, whereas for $\omega$ between black and green components, $R_B(L,\omega)$ is not empty.

However, we consider only the black, blue, and red components (ignoring the green component), we get a three component link for which if $\omega$ goes between any two components, we get $R_B(L,\omega) = \emptyset$.

\begin{figure}[H]
\centering
\includegraphics[width=100mm]{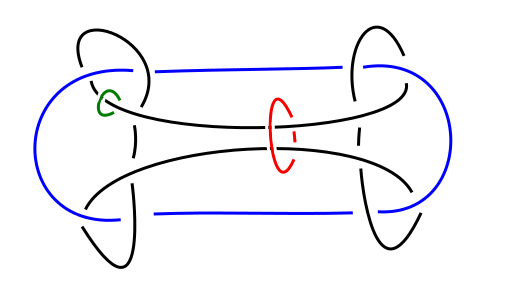}
\caption{}
\label{det0_eg}
\end{figure}

The question of whether $\omega$ has $R_B(L,\omega) = \emptyset$ or $R_B(L,\omega) = R_B(L, \emptyset)$ relates to the question of whether the corresponding elements of $\ker(M^T)$ over $\zz/2\zz$ we constructed in the proof above lift to elements of $\ker(M^T)$ over $\zz$.

\end{document}